\documentclass{amsart}
\usepackage{amssymb}
\usepackage[USenglish]{babel}
\usepackage{graphicx}
\usepackage{overpic}
\usepackage{caption}
\usepackage{microtype}
\usepackage{subcaption}
\usepackage{enumitem}

\usepackage[margin=3.5cm, marginparwidth=2.8cm]{geometry}

\usepackage[colorlinks=true, pdfstartview=FitV, linkcolor=blue, citecolor=blue, urlcolor=blue]{hyperref}
\usepackage{thm-restate}
\usepackage[super]{nth}
\usepackage{marginnote}
\usepackage{float}
\usepackage{tikzit}

\tikzstyle{Black}=[fill=black, draw=black, shape=circle]
\tikzstyle{Ghost}=[draw=none, shape=circle, fill=none]
\tikzstyle{Empty}=[fill=white, draw=black, shape=circle]
\tikzstyle{none}=[inner sep=0mm]
\tikzstyle{Black_rectangle}=[fill=black, draw=black, shape=rectangle]
\tikzstyle{Big}=[fill=white, draw=black, shape=circle, minimum size=2.4cm]

\tikzstyle{Order}=[dashed, ->, draw={rgb,255: red,191; green,191; blue,191}, ultra thick]
\tikzstyle{Axis}=[draw={rgb,255: red,191; green,191; blue,191}, -]
\tikzstyle{Dashed}=[dashed, -, draw={rgb,255: red,191; green,191; blue,191}]
\tikzstyle{Dashed->}=[->, dashed, draw={rgb,255: red,191; green,191; blue,191}]
\tikzstyle{Black arrow}=[->]

\usepackage{mathtools}

\theoremstyle{plain}
\newtheorem{thm}{Theorem}[section]
\newtheorem{lemma}[thm]{Lemma}
\newtheorem{prop}[thm]{Proposition}
\newtheorem{cor}[thm]{Corollary}

\newcounter{theoremalph}

\newtheorem{thmAlph}[theoremalph]{Theorem}
\newcounter{prob}
\newtheorem{pb}[prob]{Problem}

\numberwithin{equation}{section}
\numberwithin{figure}{section}

\theoremstyle{definition}
\newtheorem{definition}[thm]{Definition}

\newtheorem{remark}[thm]{Remark}

\newtheorem{conv}[thm]{Convention}

\newtheorem{notation}[thm]{Notation}

\usepackage{marginnote}
\newcounter{mynote}

\newcommand{\N}{\mathbb{N}}

\newcommand{\Z}{\mathbb{Z}}

\newcommand{\bN}{\mathbb{N}}
\newcommand{\bZ}{\mathbb{Z}}

\newcommand{\cC}{\mathcal{C}}

\newcommand{\cG}{\mathcal{G}}

\newcommand{\cX}{\mathcal{X}}

\DeclareMathOperator{\enc}{enc}

\DeclareMathOperator{\Exc}{Ex}

\title[Minimal sets in the plane]{Planar lattice subsets with minimal vertex boundary}
\author{Radhika Gupta}
\address{Radhika Gupta, Department of Mathematics, Temple University, Wachman Hall, 1805 North Broad Street, Philadelphia, PA 19122, USA}
\email{radhikagupta.maths@gmail.com}

\author{Ivan Levcovitz}
\address{Ivan Levcovitz, Mathematics Department, Technion - Israel Institute of Technology, Haifa, 32000, Israel}
\email{levcovitz@technion.ac.il}

\author{Alexander Margolis}
\address{Alexander Margolis, Mathematics Department, Technion - Israel Institute of Technology, Haifa, 32000, Israel}
\email{amargolis@campus.technion.ac.il}

\author{Emily Stark}
\address{Emily Stark, Department of Mathematics and Computer Science, Wesleyan University, Science Tower 655, 265 Church Street, Middletown, CT 06459, USA}
\email{estark@wesleyan.edu}


\date{\today}

\begin{document}

\begin{abstract}
A subset of vertices of a graph is \textit{minimal} if, within all subsets of the same size, its vertex boundary is minimal.
We give a complete, geometric characterization of minimal sets for the planar integer lattice $X$.
Our characterization elucidates the structure of all minimal sets, and we are able to use it to obtain several applications. 
We characterize \textit{uniquely minimal} sets of $X$: those which are congruent to any other minimal set of the same size.
We also classify all \emph{efficient} sets of $X$: those that have  maximal size amongst all such sets with a fixed vertex boundary. 
We define and investigate the \emph{graph $\cG$ of minimal sets} whose vertices are congruence classes of minimal sets of $X$ and whose edges connect vertices which can be represented by minimal sets that differ by exactly one vertex. 
We prove that $\cG$ has exactly one infinite component, has infinitely many isolated vertices and has bounded components of arbitrarily large size.
Finally, we show that all minimal sets, except one, are connected. 
\end{abstract}

\maketitle

\section{Introduction}


The classical isoperimetric problem can be stated as follows: amongst all closed curves in the plane with fixed length, characterize those that enclose the maximal area. The solution to this isoperimetric problem is the circle.
By a simple scaling argument, this problem is easily  seen to be equivalent to the following dual problem: 
\begin{pb}\label{pb:classiciso}
Amongst all closed curves in the plane that enclose a fixed area, characterize those that  have  minimal length. 
\end{pb}
\noindent The isoperimetric problem dates back to antiquity, as documented  in Virgil's account of the founding of Carthage in the  \emph{Aeneid}. However,  the first steps towards a  solution of Problem~\ref{pb:classiciso} were given relatively recently  by Steiner in the \nth{19} century. 
In this article, we give a  solution to the
discrete graph-theoretic  analogue of Problem~\ref{pb:classiciso}.

Discrete isoperimetric problems  have been studied extensively in graph theory, and there  are many
 applications in areas such as network design and the theory of error correcting codes \cite{harper-book, hoory2006expander}.
  Given a graph $X$ with vertex set $V(X)$, the \emph{vertex boundary} of $A\subset V(X)$  is defined by  
\[\partial A:= \{u \in V(X) \setminus A \mid \text{ there exists } v \in A \text{ such that } (u,v) \in E(X)\}.\] 
The \emph{vertex isoperimetric problem} for a graph $X$ is the following:
\begin{pb}\label{pb:isop}
	Amongst all subsets of $V(X)$ with a fixed number of vertices, characterize those that have  minimal size vertex boundary.
\end{pb} 
The sets that appear as solutions to Problem~\ref{pb:isop} are called \emph{minimal}.

\subsection*{The isoperimetric problem for the integer lattice in the plane}
In this article we study the graph $X=\Z^2_{\ell_1}$ with vertex set $X^0=\bZ^2$ and edges connecting all pairs of vertices $\ell_1$-distance one apart.
A nested sequence of minimal sets for $X$ is given by Wang--Wang \cite{wangwang77}.

Our approach differs from the usual one of finding a sequence of minimal sets, in that we give a geometric characterization of  \textit{every} minimal set. 
While circles are the natural geometric solution to Problem~\ref{pb:classiciso}, there can be many different congruence classes of minimal sets in $X$ of a given size and our result exactly describes these solutions.
This approach lets us prove many applications that allow us to better understand the collection of all minimal sets. 

Before describing our results, we first establish some notation. We consider subsets of $X^0$ up to the following natural equivalence relation: we say two subsets $A,B \subset X^0$  are \emph{congruent} if there is a graph automorphism $\phi$ of $X$ such that $\phi(A)=B$. It is clear that if $A$ and $B$ are congruent, then $A$ is minimal if and only if $B$ is minimal. 

 Given natural numbers $\alpha, \beta \in \bN$, we define 
$B(\alpha,\beta)$ to be  the set of vertices $(x,y)\in X^0$ that satisfy $0 \leq y-x \leq \alpha$ and $0 \leq y+x \leq \beta$. Similarly, given even integers  $\alpha, \beta \in \bN$, we define  $\hat B(\alpha,\beta)$ to be  the set of vertices $(x,y)\in X^0$ that satisfy $0 \leq y-x \leq \alpha$ and $-1 \leq y+x \leq \beta-1$. A \emph{box} is a non-empty subset of $X^0$ that is congruent to either $B(\alpha,\beta)$ or $
\hat B(\alpha,\beta)$ for some $\alpha, \beta$.
The \emph{enclosing box} of a set $A$, denoted $\enc(A)$,  is the smallest box containing $A$. Examples of sets and their enclosing boxes are shown in Figure~\ref{fig:examples of minimal sets}. 

Since boxes are  parametrized by numbers $\alpha,\beta \in \bN$, 
it is  easier to determine whether a box is minimal than it is  to determine whether an arbitrary set of vertices is minimal.
Therefore, our broad strategy in solving Problem~\ref{pb:isop} is to compare an arbitrary set to its enclosing box. We show that the enclosing box can be obtained by ``saturating'' a set, i.e. by adding vertices that do not increase the boundary. In the course of our proof, we classify precisely which boxes are minimal, see Remark~\ref{rem:min boxes}.  

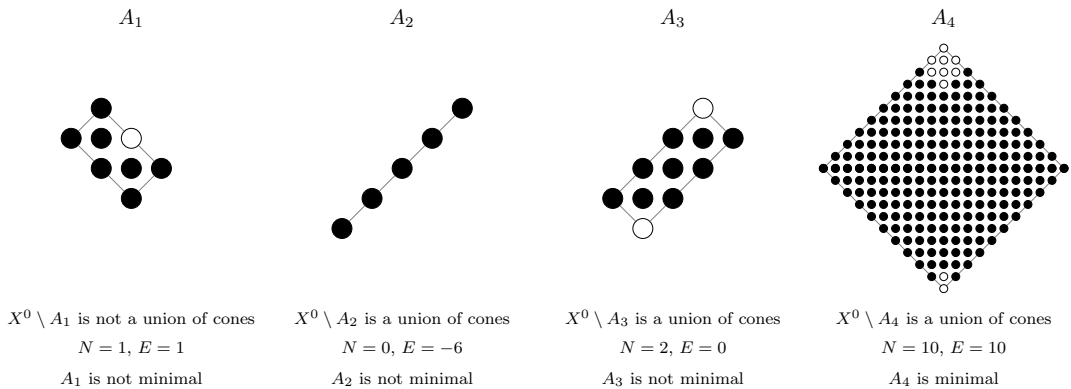
\begin{figure}[H] 
	\centering
	\scalebox{.8}{\begin{tikzpicture}
	\begin{pgfonlayer}{nodelayer}
		\node [style=Black] (0) at (0, 1) {};
		\node [style=Black] (2) at (1, 0) {};
		\node [style=Black] (3) at (0, 0) {};
		\node [style=Black] (4) at (0, 2) {};
		\node [style=Black] (5) at (-1, 1) {};
		\node [style=Black] (6) at (1, -1) {};
		\node [style=Black] (7) at (2, 0) {};
		\node [style=none] (8) at (1, -5) {};
		\node [style=none] (9) at (1, -5) {\footnotesize{$X^0\setminus A_1$ is not a union of cones}};
		\node [style=Empty] (10) at (1, 1) {};
		\node [style=none] (11) at (1, -6) {\footnotesize{$N=1$, $E=1$}};
		\node [style=none] (12) at (1, -7) {\footnotesize{$A_1$ is not minimal}};
		\node [style=none] (55) at (28, -5) {\footnotesize{$X^0\setminus A_4$ is a union of cones}};
		\node [style=none] (56) at (28, -6) {\footnotesize{$N=10$, $E=10$}};
		\node [style=none] (57) at (28, -7) {\footnotesize{$A_4$ is minimal}};
		\node [style=none]  at (10, -5) {\footnotesize{$X^0\setminus A_2$ is  a union of cones}};
		\node [style=none] at (10, -6) {\footnotesize{$N=0$, $E=-6$}};
		\node [style=none] at (10, -7) {\footnotesize{$A_2$ is not minimal}};
		\node [style=none] at (10, 5) {$A_2$};
		\node [style=none] (58) at (28, 5) {$A_4$};
		\node [style=none] (59) at (1, 5) {$A_1$};
		\node [style=Black] (67) at (19, -1) {};
		\node [style=Black] (68) at (20, 0) {};
		\node [style=Black] (69) at (21, 1) {};
		\node [style=Black] (76) at (18, -1) {};
		\node [style=Black] (77) at (19, 0) {};
		\node [style=Black] (78) at (20, 1) {};
		\node [style=Black] (87) at (17, -1) {};
		\node [style=Black] (88) at (18, 0) {};
		\node [style=Black] (89) at (19, 1) {};
		\node [style=none] (101) at (19, -5) {};
		\node [style=none] (102) at (19, -5) {\footnotesize{$X^0\setminus A_3$ is  a union of cones}};
		\node [style=none] (103) at (19, -6) {\footnotesize{$N=2$, $E=0$}};
		\node [style=none] (104) at (19, -7) {\footnotesize{$A_3$ is not minimal}};
		\node [style=none] (105) at (19, 5) {$A_3$};
		\node [style=Empty] (106) at (20, 2) {};
		\node [style=Empty] (107) at (18, -2) {};
	\end{pgfonlayer}
	\begin{pgfonlayer}{edgelayer}
		\draw [color=gray] (4) to (7);
		\draw [color=gray] (7) to (6);
		\draw [color=gray] (6) to (5);
		\draw [color=gray] (5) to (4);
		\draw [color=gray] (87) to (107);
		\draw [color=gray] (107) to (69);
		\draw [color=gray] (106) to (69);
		\draw [color=gray] (106) to (87);
	\end{pgfonlayer}
	\draw [color=gray] (28-4,0) to (28,4);
	\draw [color=gray] (28+4,0) to (28,4);
	\draw [color=gray] (28-4,0) to (28,-4);
	\draw [color=gray] (28+4,0) to (28,-4);
	\foreach \y in {9,...,10}{
		\pgfmathsetmacro\startx{\y-10}
		\pgfmathsetmacro\endx{10-\y} 
		\foreach \x in {\startx,...,\endx}
		{	
			\node [style=Empty,scale=0.4]  at (0.4*\x+28, 0.4*\y) {};
		}
	}
	\foreach \y in {8}{
		\pgfmathsetmacro\startx{\y-10}
		\pgfmathsetmacro\endx{10-\y} 
		\foreach \x in {-1,0,1}
		{	
			\node [style=Empty,scale=0.4]  at (0.4*\x+28, 0.4*\y) {};
		}
		\foreach \x in {-2,2}
		{	
			\node [style=Black,scale=0.4]  at (0.4*\x+28, 0.4*\y) {};
		}
	}
\foreach \y in {7}{
	\pgfmathsetmacro\startx{\y-10}
	\pgfmathsetmacro\endx{10-\y} 
	\foreach \x in {\startx,...,-1}
	{	
		\node [style=Black,scale=0.4]  at (0.4*\x+28, 0.4*\y) {};
	}
	\foreach \x in {1,...,\endx}
	{	
		\node [style=Black,scale=0.4]  at (0.4*\x+28, 0.4*\y) {};
	}
	\foreach \x in {0}
	{	
		\node [style=Empty,scale=0.4]  at (0.4*\x+28, 0.4*\y) {};
	}
}
	\foreach \y in {0,...,6}{
		\pgfmathsetmacro\startx{\y-10}
		\pgfmathsetmacro\endx{10-\y} 
		\foreach \x in {\startx,...,\endx}
		{	
			\node [style=Black,scale=0.4]  at (0.4*\x+28, 0.4*\y) {};
		}
	}
	\foreach \y in {-8,...,0}{
		\pgfmathsetmacro\startx{-\y-10}
		\pgfmathsetmacro\endx{\y+10} 
		\foreach \x in {\startx,...,\endx}
		{	
			\node [style=Black,scale=0.4]  at (0.4*\x+28, 0.4*\y) {};
		}
	}
	\foreach \y in {-9}{
		\pgfmathsetmacro\startx{-\y-10}
		\pgfmathsetmacro\endx{\y+10} 
		\foreach \x in {0}
		{	
			\node [style=Empty,scale=0.4]  at (0.4*\x+28, 0.4*\y) {};
		}
		\foreach \x in {-1,1}
		{	
			\node [style=Black,scale=0.4]  at (0.4*\x+28, 0.4*\y) {};
		}
	}
	\foreach \y in {-10}{
		\pgfmathsetmacro\startx{-\y-10}
		\pgfmathsetmacro\endx{\y+10} 
		\foreach \x in {0}
		{	
			\node [style=Empty,scale=0.4]  at (0.4*\x+28, 0.4*\y) {};
		}
	}

	\draw [color=gray] (+10-2,-2) to (+10+2,2);
	\foreach \y in {-2,...,2}{
	\pgfmathsetmacro\startx{\y}
	\pgfmathsetmacro\endx{\y} 
	\foreach \x in {\startx,...,\endx}
	{	
		\node [style=Black,scale=1]  at (\x+10, \y) {};
	}
	
}

\end{tikzpicture}}
	
	\caption{Examples of sets $A_i\subset X^0$ and their enclosing boxes. Black vertices are contained in $A_i$ and white vertices are contained in $\enc(A_i)\setminus A_i$.}
	\label{fig:examples of minimal sets}
\end{figure}

A {\it cone} is a subset of $X^0$ congruent to the set $\{(x,y) \, \mid \, y-x \geq 0, y+x \geq 0\}$.
As a precursor to our main result, we give a necessary condition for minimality: if a set $A$ is minimal, then  its complement $X^0\setminus A$ is a union of cones and furthermore, its enclosing box $\enc(A)$ is also minimal.
Although this is  far from a complete classification of minimal sets --- which we give in Theorem~\ref{thm:minimal set characterization intro} --- it demonstrates the important role of the enclosing box in determining minimality.

If $A_1$ is the set shown in   Figure~\ref{fig:examples of minimal sets}, its complement $X^0\setminus A_1$ is not a union of cones. This follows as any cone containing the  white vertex in $\enc(A_1)\setminus A_1$ must also contain a vertex of $A_1$, and so $X^0\setminus A_1$ is not a union of cones. Thus $A_1$ is not minimal by the preceding  necessary condition. Similarly, $A_2$ can be seen  not to be  minimal since  its enclosing box $\enc(A_2)$ is not minimal. 
However, to say whether or not $A_3$ and $A_4$ are minimal is a slightly  more delicate matter, since both $A_3$ and $A_4$ satisfy the preceding necessary condition for minimality. To see why $A_3$ is not minimal and $A_4$ is, we  use a numerical invariant of a box called its excess. 

Denoted $\Exc(B)$,  the \emph{excess} of a box $B$ measures how much larger a box is than the smallest minimal set $A$ with $\lvert\partial A\rvert=\lvert\partial B\rvert$; see Definition~\ref{def:excess}. In particular, the excess of a box is non-negative if and only if the box is minimal.
An explicit formula for the excess of a box is given in Theorem~\ref{thm:box_excess}. If a box has width $r-k$ and length $r+k$, then its excess is approximately $\frac{r-k^2}{2}$.
Thus, boxes that are sufficiently close to being squares have positive excess, whilst boxes that are sufficiently long and narrow have negative excess.

Our main theorem, stated below, gives two related characterizations of minimal sets in terms of their enclosing boxes. 
\newcommand{\mainthm}{

}
\begin{thmAlph}[Theorem \ref{thm:minimal set characterization_refined}]\label{thm:minimal set characterization intro}
	Let $A\subset X^0$ with $N\coloneqq \lvert \enc(A)\setminus A \rvert$ and $E:=\Exc(\enc(A))$. Then the following are equivalent:
	\begin{enumerate}
		\item $A$ is minimal;
		\item $\lvert \partial A\rvert = \lvert \partial (\enc(A))\rvert$ and $N \le E$;
		\item $X^0\setminus A$ is a union of cones and $N \le E$.
	\end{enumerate} 
\end{thmAlph}
Going back to the examples in Figure~\ref{fig:examples of minimal sets}, a straightforward application of Theorem~\ref{thm:box_excess} --- the formula for the excess of a box --- tells us that $A_3$ does not satisfy $N\le E$, but  $A_4$ does. Since $X^0\setminus A_3$ and $X^0\setminus A_4$ are both unions of cones,  Theorem~\ref{thm:minimal set characterization intro} can be used to deduce that $A_3$ is not minimal and $A_4$ is  minimal.

\subsection*{Applications}

A natural question to consider is  whether minimal sets of a fixed size are  unique up to congruence. 
	More formally, $A\subset X^0$ is \emph{uniquely minimal} if $A$ is minimal  and any  minimal set containing the same number of vertices as $A$ is congruent to $A$.
 We  completely classify uniquely minimal sets:
 \newcommand{\uniquemin}{A subset of $X^0$ is uniquely minimal if and only if it is congruent to either $B(2n, 2n)$ or $B( n, n+1)$ for some $n\in \bN$.  }
\begin{thmAlph}[Theorem~\ref{thm:unique_min}]\label{thm:unique_min_intro}
	\uniquemin
\end{thmAlph}

As well as understanding individual minimal sets, we also  want to understand the structure of the collection of all minimal sets. To do this, we initiate the  study of the \emph{graph $\mathcal{G}$ of minimal sets}. Vertices of $\mathcal{G}$ are congruence classes  of minimal sets. Two vertices $v$ and $w$ in $\cG$ are joined by an edge if there exist representative minimal sets  $A\in v$ and $B\in w$ whose symmetric difference  has size one. The graph $\mathcal{G}$ has a natural grading corresponding to the sizes of representative minimal sets. The induced subgraphs of $\cG$ containing all  congruence classes of minimal sets of size at most 10 and 41 are shown in Figures~\ref{fig:graph} and~\ref{fig:graph_41} respectively.

\begin{figure}[p] 
	\vspace{.5in}
	\centering
	\scalebox{.4}{\begin{tikzpicture}
	\begin{pgfonlayer}{nodelayer}
		\node [style=Big] (0) at (0, 0) {};
		\node [style=none] (154) at (0, 3.25) {\huge $B(0,0)$};
		\node [style=Black] (1) at (0, 0) {};
		\node [style=Big] (2) at (7, 3) {};
		\node [style=none] (152) at (7, 6.25) {\huge $B(1,1)$};
		\node [style=none] (153) at (7, -6.25) {\huge $B(0,2)$};
		\node [style=Big] (3) at (7, -3) {};
		\node [style=Black] (4) at (7, 3) {};
		\node [style=Black] (5) at (7, 4) {};
		\node [style=Black] (6) at (7, -3) {};
		\node [style=Black] (7) at (8, -4) {};
		\node [style=Big] (8) at (14, 0) {};
		\node [style=none] (151) at (14, 3.25) {\huge $B(1,2)$};
		\node [style=Black] (9) at (14, 0) {};
		\node [style=Black] (10) at (14, 1) {};
		\node [style=Black] (11) at (15, 0) {};
		\node [style=Big] (12) at (21, -6) {};
		\node [style=none] (150) at (21, -9.25) {\huge $\hat B(2,2)$};
		\node [style=Black] (13) at (20.5, -5.5) {};
		\node [style=Black] (14) at (20.5, -6.5) {};
		\node [style=Black] (15) at (21.5, -5.5) {};
		\node [style=Black] (16) at (21.5, -6.5) {};
		\node [style=Big] (17) at (21, 6) {};
		\node [style=none] (149) at (21, 9.25) {\huge $B(1,3)$};
		\node [style=Big] (18) at (21, 0) {};
		\node [style=Black] (19) at (20.5, 6) {};
		\node [style=Black] (20) at (21.5, 6) {};
		\node [style=Black] (21) at (20.5, 7) {};
		\node [style=Black] (22) at (21.5, 5) {};
		\node [style=Black] (23) at (21, 0) {};
		\node [style=Black] (24) at (21, 1) {};
		\node [style=Black] (25) at (21, -1) {};
		\node [style=Black] (26) at (22, 0) {};
		\node [style=Big] (27) at (28, 0) {};
		\node [style=none] (148) at (28, 3.25) {\huge $B(2,2)$};
		\node [style=Black] (28) at (28, 0) {};
		\node [style=Black] (29) at (28, 1) {};
		\node [style=Black] (30) at (27, 0) {};
		\node [style=Black] (31) at (28, -1) {};
		\node [style=Black] (32) at (29, 0) {};
		\node [style=Big] (33) at (35, 0) {};
		\node [style=none] (147) at (35, 3.25) {\huge $B(2,3)$};
		\node [style=Black] (34) at (35, 0) {};
		\node [style=Black] (35) at (35, 1) {};
		\node [style=Black] (36) at (34, 0) {};
		\node [style=Black] (37) at (35, -1) {};
		\node [style=Black] (38) at (36, 0) {};
		\node [style=Black] (39) at (36, 1) {};
		\node [style=Big] (40) at (43, 9) {};
		\node [style=Black] (41) at (43, 8.5) {};
		\node [style=Black] (42) at (43, 9.5) {};
		\node [style=Black] (43) at (42, 8.5) {};
		\node [style=Black] (44) at (43, 7.5) {};
		\node [style=Black] (45) at (44, 8.5) {};
		\node [style=Black] (46) at (44, 9.5) {};
		\node [style=Black] (47) at (44, 10.5) {};
		\node [style=Big] (48) at (43, 3) {};
		\node [style=Black] (49) at (43, 2.5) {};
		\node [style=Black] (50) at (43, 3.5) {};
		\node [style=Black] (51) at (42, 2.5) {};
		\node [style=Black] (52) at (43, 1.5) {};
		\node [style=Black] (53) at (44, 2.5) {};
		\node [style=Black] (54) at (44, 3.5) {};
		\node [style=Black] (55) at (43, 4.5) {};
		\node [style=Big] (56) at (43, -3) {};
		\node [style=Black] (57) at (43, -3.5) {};
		\node [style=Black] (58) at (43, -2.5) {};
		\node [style=Black] (59) at (42, -3.5) {};
		\node [style=Black] (60) at (43, -4.5) {};
		\node [style=Black] (61) at (44, -3.5) {};
		\node [style=Black] (62) at (44, -2.5) {};
		\node [style=Black] (63) at (42, -2.5) {};
		\node [style=Big] (64) at (43, -9) {};
		\node [style=Black] (65) at (43, -9) {};
		\node [style=Black] (66) at (43, -8) {};
		\node [style=Black] (67) at (42, -9) {};
		\node [style=Black] (68) at (43, -10) {};
		\node [style=Black] (69) at (44, -9) {};
		\node [style=Black] (70) at (44, -8) {};
		\node [style=Black] (71) at (42, -10) {};
		\node [style=none] (146) at (43, -12.25) {\huge $\hat B(2,4)$};
		\node [style=Big] (72) at (49, 9) {};
		\node [style=Black] (73) at (48.5, 8.5) {};
		\node [style=Black] (74) at (48.5, 9.5) {};
		\node [style=Black] (75) at (47.5, 8.5) {};
		\node [style=Black] (76) at (48.5, 7.5) {};
		\node [style=Black] (77) at (49.5, 8.5) {};
		\node [style=Black] (78) at (49.5, 9.5) {};
		\node [style=Black] (79) at (49.5, 10.5) {};
		\node [style=Black] (80) at (50.5, 9.5) {};
		\node [style=none] (145) at (49, 12.25) {\huge $B(2,4)$};
		\node [style=Big] (81) at (49, 0) {};
		\node [style=Black] (82) at (49, -0.5) {};
		\node [style=Black] (83) at (49, 0.5) {};
		\node [style=Black] (84) at (48, -0.5) {};
		\node [style=Black] (85) at (49, -1.5) {};
		\node [style=Black] (86) at (50, -0.5) {};
		\node [style=Black] (87) at (50, 0.5) {};
		\node [style=Black] (88) at (48, 0.5) {};
		\node [style=Black] (89) at (49, 1.5) {};
		\node [style=none] (144) at (49, 3.25) {\huge $B(3,3)$};
		\node [style=Big] (90) at (56, 12) {};
		\node [style=Black] (91) at (55.25, 11.5) {};
		\node [style=Black] (92) at (55.25, 12.5) {};
		\node [style=Black] (93) at (54.25, 11.5) {};
		\node [style=Black] (94) at (55.25, 10.5) {};
		\node [style=Black] (95) at (56.25, 11.5) {};
		\node [style=Black] (96) at (56.25, 12.5) {};
		\node [style=Black] (97) at (56.25, 13.5) {};
		\node [style=Black] (98) at (57.25, 12.5) {};
		\node [style=Black] (99) at (57.25, 13.5) {};
		\node [style=none] (143) at (56, 15.25) {\huge $B(2,5)$};
		\node [style=Big] (100) at (56, 0) {};
		\node [style=Black] (101) at (55.5, -0.5) {};
		\node [style=Black] (102) at (55.5, 0.5) {};
		\node [style=Black] (103) at (54.5, -0.5) {};
		\node [style=Black] (104) at (55.5, -1.5) {};
		\node [style=Black] (105) at (56.5, -0.5) {};
		\node [style=Black] (106) at (56.5, 0.5) {};
		\node [style=Black] (107) at (54.5, 0.5) {};
		\node [style=Black] (108) at (55.5, 1.5) {};
		\node [style=Black] (109) at (57.5, 0.5) {};
		\node [style=Big] (110) at (56, 6) {};
		\node [style=Black] (111) at (55.5, 5.5) {};
		\node [style=Black] (112) at (55.5, 6.5) {};
		\node [style=Black] (113) at (54.5, 5.5) {};
		\node [style=Black] (114) at (55.5, 4.5) {};
		\node [style=Black] (115) at (56.5, 5.5) {};
		\node [style=Black] (116) at (56.5, 6.5) {};
		\node [style=Black] (117) at (56.5, 7.5) {};
		\node [style=Black] (118) at (57.5, 6.5) {};
		\node [style=Black] (119) at (56.5, 4.5) {};
		\node [style=Big] (120) at (56, -6) {};
		\node [style=Black] (121) at (55.75, -6.5) {};
		\node [style=Black] (122) at (55.75, -5.5) {};
		\node [style=Black] (123) at (54.75, -6.5) {};
		\node [style=Black] (124) at (55.75, -7.5) {};
		\node [style=Black] (125) at (56.75, -6.5) {};
		\node [style=Black] (126) at (56.75, -5.5) {};
		\node [style=Black] (127) at (54.75, -5.5) {};
		\node [style=Black] (128) at (55.75, -4.5) {};
		\node [style=Black] (129) at (56.75, -7.5) {};
		\node [style=Big] (130) at (62, 0) {};
		\node [style=Black] (131) at (61.75, -0.5) {};
		\node [style=Black] (132) at (61.75, 0.5) {};
		\node [style=Black] (133) at (60.75, -0.5) {};
		\node [style=Black] (134) at (61.75, -1.5) {};
		\node [style=Black] (135) at (62.75, -0.5) {};
		\node [style=Black] (136) at (62.75, 0.5) {};
		\node [style=Black] (137) at (60.75, 0.5) {};
		\node [style=Black] (138) at (61.75, 1.5) {};
		\node [style=Black] (139) at (62.75, -1.5) {};
		\node [style=Black] (140) at (63.75, -0.5) {};
		\node [style=none] (142) at (62, 3.25) {\huge $B(3,4)$};
	\end{pgfonlayer}
	\begin{pgfonlayer}{edgelayer}
		\draw (1) to (4);
		\draw (1) to (6);
		\draw (4) to (9);
		\draw (9) to (6);
		\draw (9) to (17);
		\draw (9) to (12);
		\draw (9) to (18);
		\draw (27) to (23);
		\draw (28) to (34);
		\draw (40) to (34);
		\draw (48) to (34);
		\draw (56) to (34);
		\draw (65) to (34);
		\draw (81) to (48);
		\draw (81) to (56);
		\draw (40) to (72);
		\draw (72) to (90);
		\draw (81) to (100);
		\draw (110) to (72);
		\draw (120) to (81);
		\draw (130) to (120);
		\draw (100) to (130);
		\draw (110) to (130);
	\end{pgfonlayer}
\end{tikzpicture}}
	
	\caption{The induced subgraph of $\cG$ containing congruence classes of minimal sets of size at most 10. All boxes are labelled using Notation~\ref{notation:boxes}.}\label{fig:graph}
\end{figure}
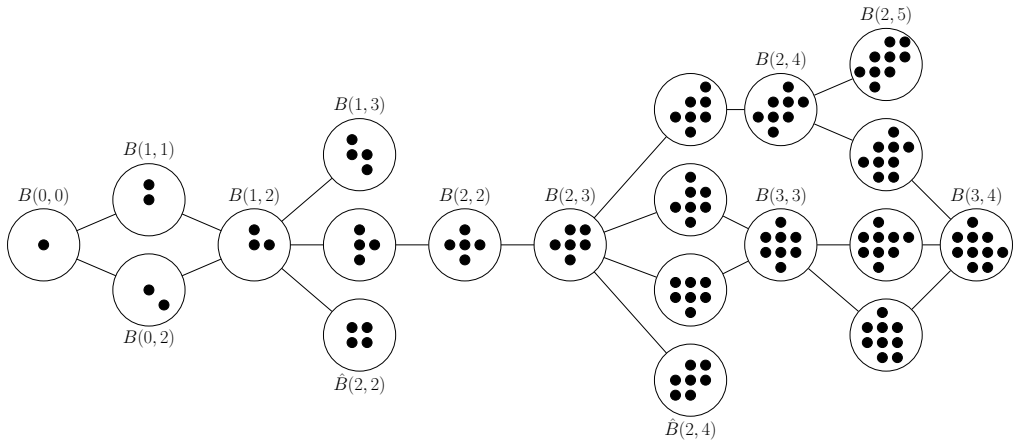

\begin{figure}[p]
	\vspace{.5in}
	\begin{subfigure}[t]{0.47\textwidth}
		\includegraphics[scale=.34]{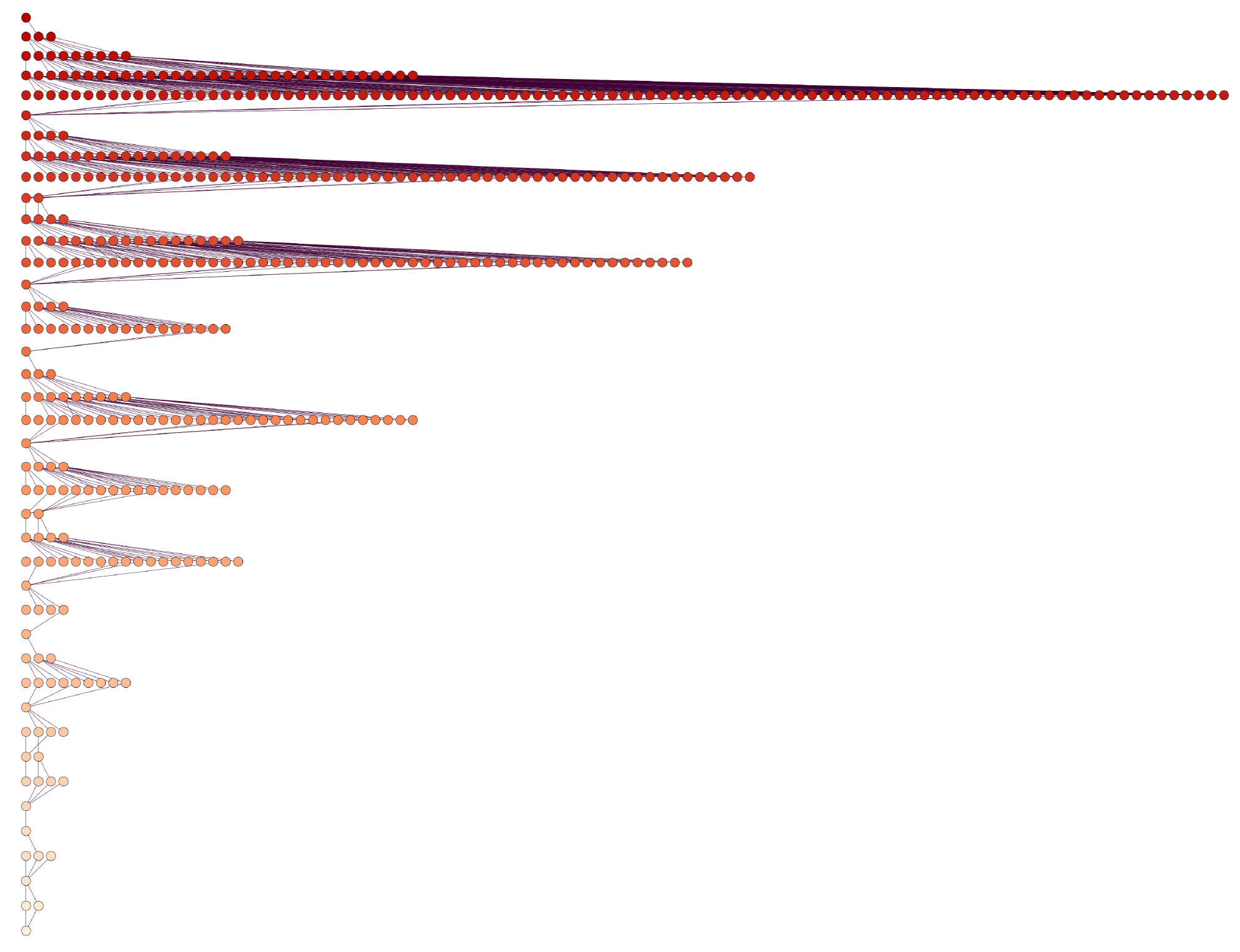}
		\caption{
		The induced subgraph of $\mathcal{G}$ of congruence classes of minimal sets of size up to 41. The graded structure of the graph is shown where vertices representing sets of larger sizes appear above and are shaded with a darker color than those representing smaller sizes.
		}
	\end{subfigure}
	\hfill
	\begin{subfigure}[t]{0.47\textwidth}
		\includegraphics[scale=.32]{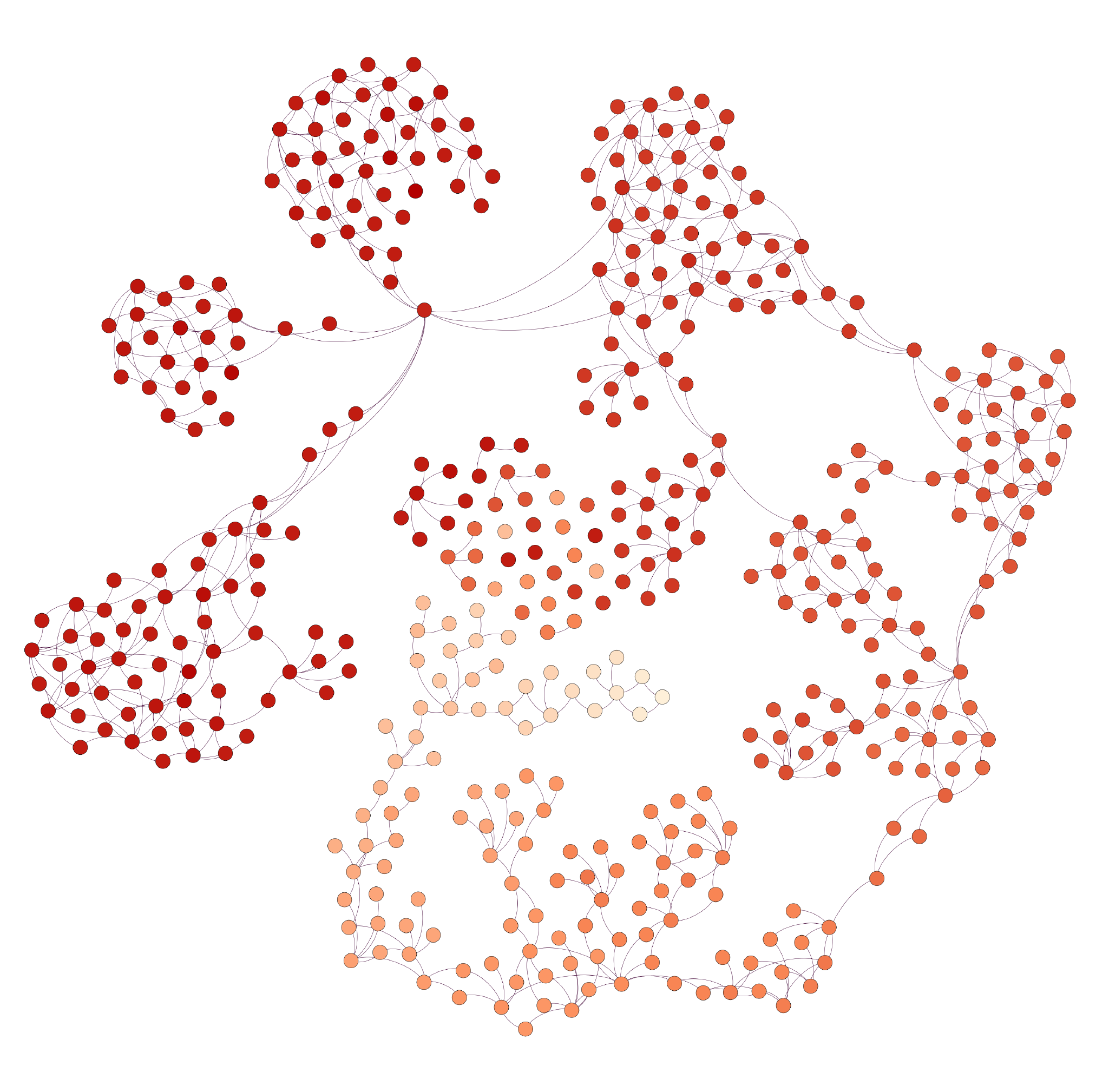}
		\caption{
		The same graph as on the left displayed here in a different layout. Isolated vertices and finite size components are towards the center of the graph. The darkness of the shading of vertices is still proportional to the size of the corresponding sets.
		}
	\end{subfigure}
	\caption{}
	\label{fig:graph_41}
\end{figure}

We exhibit the following features of $\mathcal{G}$: 
\begin{thmAlph}
Let $\mathcal{G}$ be the graph of minimal sets. Then:
\begin{enumerate}
	\item $\cG$ contains a single infinite component (Corollary~\ref{cor_one_inf_comp});
	\item $\cG$ contains finite components of arbitrarily large height (Theorem~\ref{cor:large_comps}), where the \textit{height} of a component is the maximal length of a nested sequence of minimal sets in it;
	\item $\cG$ contains infinitely many isolated vertices (Theorem~\ref{cor:large_comps}).
\end{enumerate}  
\end{thmAlph}

We note that the  infinite sequence of nested minimal sets constructed by Wang--Wang is contained in the unique infinite component of $\mathcal{G}$. 
\vspace{.4cm} 

At the beginning of the introduction we mentioned two equivalent formulations of the isoperimetric problem in the Euclidean plane: maximizing the \emph{area} enclosed by a  curve of fixed \emph{length}, or minimising the \emph{length} of a curve enclosing a  fixed \emph{area}. We thus consider the following discrete isoperimetric problem dual to Problem~\ref{pb:isop}:
\begin{pb}\label{pb:isopdual}
	Amongst all subsets of $X^0$ with a fixed vertex boundary, characterize those with maximal size.
\end{pb}
We say a subset of $X^0$ is an \emph{efficient set} if it is a  solution to Problem~\ref{pb:isopdual}. It turns out that Problems~\ref{pb:isop} and~\ref{pb:isopdual} are \emph{not} equivalent for the integer lattice. Every efficient set is minimal (see Lemma~\ref{lem:dead_sat+box}), but it is not the case that every minimal set is efficient. We  give an explicit solution to Problem~\ref{pb:isopdual}.
\begin{thmAlph}[Lemma~\ref{lem:boxefficient}]\label{thm:efficient_intro}
	A subset of $X^0$ is efficient if and only if  it is congruent to either $B(n,n)$, $B(n,n+1)$ or $B(2n,2n+2)$ for some $n\in \bN$.
\end{thmAlph}

While writing this article, we learned that Theorem~\ref{thm:efficient_intro} was essentially already known.
Vainsencher--Bruckstein characterize sets that  are both efficient and minimal, which they call \emph{Pareto optimal} sets \cite{vainsencher2008isop}. However, by Wang--Wang's result and an easy argument, it follows that efficient sets are always minimal, giving the above theorem. We note that our proof is independent of that of Vainsencher--Bruckstein.
\vspace{.4cm}

The \emph{life expectancy} of a minimal set $A\subset X^0$ is defined to be the maximum 
over all $n$ such that there exists nested sequence $(A_i)_{i=0}^n$ of minimal sets with $\lvert A_i\rvert =\lvert A\rvert +i$ and $A_0=A$. The life expectancy takes values in $\bN\cup \{\infty\}$. A minimal set is said to be \emph{mortal} if it has finite life expectancy and \emph{immortal} otherwise. A minimal set is said to be \emph{dead} if its life expectancy is zero. 
We completely characterize mortal and dead sets:
\begin{thmAlph}Let  $A\subset X^0$ be a minimal set. Then:
	\begin{enumerate}
		\item 
		$A$ is dead  if and only if it is a box that is not efficient (Theorem~\ref{thm:dead_char2}). 
		\item $A$ is mortal if and only if its enclosing box is dead (Proposition~\ref{thm:mortalchar}).
	\end{enumerate}
\end{thmAlph}
We note that by Theorem~\ref{thm:efficient_intro} we get an explicit characterization of mortal and dead sets in terms of box parametrizations.
\vspace{.4cm}

A set $A\subset X^0$ is \emph{connected} if its induced subgraph is connected. It can be seen in Figure~\ref{fig:graph} that the box $B(0,2)$ is minimal but is not connected. 
We show that, up to congruence,  this is the only minimal set that is not connected.
 \newcommand{\connected}{A minimal set in $X$ is connected if and only if it is not congruent to $B(0,2)$.}
\begin{thmAlph}[Theorem~\ref{thm:connected}]
\connected
\end{thmAlph}

\subsection*{Other Related Works} \label{subsec:other_works}

A complete solution to Problem~\ref{pb:isop} is known for very few graphs.  
Much of the literature has focused on   
exhibiting a \emph{sequence  of minimal sets}, i.e. a sequence  $(A_n)$ where  each $A_n$ is a minimal set consisting of exactly $n$ vertices.
Finding such a sequence is NP-hard for a general graph (see \cite{harper-book})
and such sequences can  typically be described only  in special cases for graphs with an abundance of symmetry. 

Harper exhibited a nested sequence of minimal sets  for the $d$-dimensional hypercube $Q_d$, where $Q_d$ is a graph on the vertex set $\{0,1\}^d$ with an edge between a pair of binary strings that differ in a single coordinate  \cite{harper66}.  
This result was extended to $(P_q)^d$, the $d$-fold product of paths on $q$ vertices  \cite{chvatalova, moghadam, bollobasleader91} 
and to $(K_q)^d$, the $d$-fold product of complete graph on $q$ vertices  \cite{harper99}. In constrast with $Q_d$, there is no nested sequence of minimal sets for $(K_q)^d$.   

The isoperimetric problem has also been studied on infinite graphs, including integer lattices. Let $\Z^d_{\ell_1}$ (respectively $\Z^d_{\ell_\infty}$) be the graph on the vertex set $\Z^d$ where two vertices are joined by an edge if their $\ell_1$-distance (respectively $\ell_\infty$-distance) is 1.
As already mentioned, Wang--Wang \cite{wangwang77}  exhibit a nested sequence of minimal sets in $\Z^d_{\ell_1}$. Sieben gives a formula  for the size of the vertex boundary of a minimal set of size $n$ in this graph under the assumption that such sets are connected \cite{sieben2008polyominoes}. This is then used to analyze strategies for what are called polyomino achievement games.  
Radcliffe--Veomett obtain a nested sequence of minimal sets in $\Z^d_{\ell_\infty}$ \cite{radcliffeveomett12}. 

The
\emph{edge boundary} of a subset of a graph is defined to be the set of edges that are incident  to both a  vertex of this subset and to a vertex outside this subset. The edge isoperimetric problem has also been well-studied for the various graphs mentioned above,
 namely, by \cite{harper64, lindsey, bernstein, hart} for $Q_d$, by \cite{lindsey} for $(K_q)^d$ and \cite{bollobasleader91edge} for $(P_q)^d$ and $\Z^d_{\ell_1}$.
Recently, \cite{barbererde} studied the edge isoperimetric problem for $\Z^d_{\ell_\infty}$ and $\Z^d$ with respect to any Cayley graph.  

Many of these preceding results use a technique called \emph{compression} or \emph{normalization} that replaces a vertex set $A\subset V(X)$ with a set $c(A)\subset V(X)$ such that $\lvert A\rvert=\lvert c(A)\rvert$ and $\lvert \partial A\rvert\geq \lvert \partial (c(A))\rvert$; see \cite{harper66}. Whilst this technique is well-suited to finding a sequence of minimal sets, it does not  generally allow one to give a structural characterization of all minimal sets.

\subsection*{Outline} In Section~\ref{sec_wangwang} we review the sequence of minimal sets constructed by Wang--Wang \cite{wangwang77}. In Section~\ref{sec:saturated} we introduce the notion of a saturated set. In Sections~\ref{sec:Excess} and~\ref{sec:box_excess} we define the excess of a  set and give an explicit formula for the excess of a box (Theorem~\ref{thm:box_excess}). In Section~\ref{sec:characterization} we show in Proposition~\ref{prop:finminsat -> box} that all saturated minimal sets are boxes. Combined with our formula for the excess of box, we prove the first part of Theorem~\ref{thm:minimal set characterization intro}, thus characterizing all minimal sets in terms of their enclosing boxes. In Section~\ref{sec:connected} we show that up to congruence, there is a unique disconnected minimal set, and we prove the second part of Theorem~\ref{thm:minimal set characterization intro}. In Section~\ref{sec:graphofminsets} we study the graph $\cG$ and classify which sets  are efficient, uniquely minimal,   dead and mortal.

\subsection*{Acknowledgements} The authors are thankful for helpful discussions with Joseph Briggs, who introduced us to Harper's Theorem. We also thank Nir Lazarovich for helpful comments and suggestions.

RG was supported by Israel Science Foundation Grant 1026/15 and EPSRC grant\linebreak EP/R042187/1. IL was supported by the Israel Science Foundation and in part by a Technion fellowship. AM was supported by the Israel Science Foundation Grant No. 1562/19. ES was supported by the Azrieli Foundation, was supported in part at the Technion by a Zuckerman Fellowship, and was partially supported by NSF RTG grant $\#$1849190.

\section{Wang--Wang sets} \label{sec_wangwang}

\emph{We recall the nested sequence of minimal sets in $X$ constructed by Wang--Wang.}
\vspace{.3cm}
\begin{figure}[b]
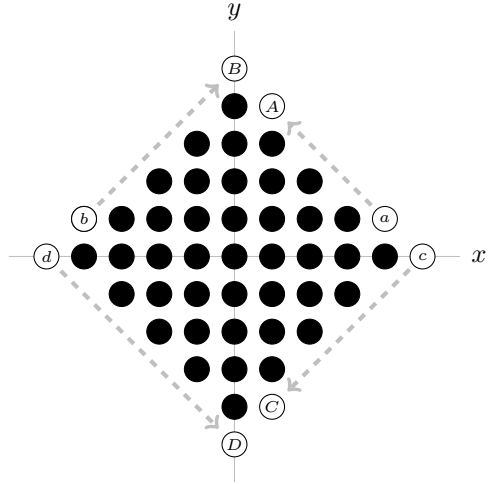
	
	\ctikzfig{images/order} 
	\caption{The black vertices denote the $\ell_1$--ball $WW_{r(n)}$ for some integer $n \ge 2$. The oriented lines show how one obtains the vertices $x_{r(n)+1}, \dots, x_{r(n+1)}$.} 
	\label{fig:order}
\end{figure}

Throughout this article, we fix the graph $X$ with vertex set $X^0=\bZ^2$, where two vertices $(x,y),(x',y')\in X^0$ are joined by an edge if and only if $\lvert x-x'\rvert+\lvert y-y'\rvert=1$.

Wang--Wang gave a nested sequence, $WW_1 \subset WW_2 \subset
\dots$, of minimal sets in $X$ such that ${|WW_n| = n}$
for all $n \ge 1$ \cite{wangwang77}.
Throughout this article, a \emph{Wang--Wang set} is a subset $A\subset X^0$ that is congruent to $WW_n$ for some $n$.
In the upcoming sections, we utilize them to prove our characterization of minimal sets in $X$.

In order to define the sets $(WW_i)$, it is enough to define a sequence of vertices $(x_i)$
for $1 \le i \le n$ such that $WW_n = \{x_1, \dots, x_n\}$. 
The first five vertices in this sequence are given by coordinates: 
\[x_1 = (0,0), ~x_2 = (1, 0), ~x_3 = (0,1), ~x_4 = (-1, 0), ~x_5 = (0, -1)\]   
Note that $\{x_1, \dots, x_5\}$ is the $\ell_1$-ball in $X$ of radius $1$ centered at $x_1$.
Let $r(n)\coloneqq 2n^2+2n+1$ denote the size of an $\ell_1$-ball in $X$ of radius $n$. 
Suppose that the vertices $x_1, \dots, x_{r(n)}$ have already been defined and
that $WW_{r(n)} = \{x_1, \dots, x_{r(n)} \}$ is the $\ell_1$--ball of radius $n$ 
centered at $x_1$, i.e., $WW_{r(n)}=\{x\in X\mid \lvert x\rvert \leq n\}$ (where $|.|$ is the $\ell_1$--norm).

We use Figure~\ref{fig:order} to define the vertices $x_{r(n)+1}, \dots,
x_{r(n+1)}$. We first set $x_{r(n)+1}$ to be the specific vertex adjacent to
$WW_{r(n)}$ shown as vertex $a$ in Figure~\ref{fig:order}. The vertices
$x_{r(n)+1}, \dots, x_{r(n) + n}$ are those along the oriented line segment
$\protect\overrightarrow{aA}$ ordered by the given  orientation.
The next set of vertices are those along the segment $\protect\overrightarrow{bB}$, then those on $\protect\overrightarrow{cC}$, and finally those on $\protect\overrightarrow{dD}$ (where each such sequence is again ordered by the given orientation).

The following lemma follows immediately  from Wang--Wang's result:
\begin{lemma}\label{lem:bdry monotone}
Let $A,B\subset X^0$. If $A$ is minimal and $\lvert A\rvert \leq\lvert B \rvert$, then $\lvert \partial A\rvert \leq \lvert \partial B\rvert$. If in addition  $\lvert \partial A\rvert = \lvert \partial B\rvert$,  then $B$ is minimal.
\end{lemma}
\begin{proof}
For every $m\geq 2$, it is easy to verify that either $\left| \partial WW_{m+1}\right| = \left| \partial WW_{m}\right|$ or $\left| \partial WW_{m+1}\right| = \left| \partial WW_{m}\right|+1$. It follows that $\lvert\partial  WW_n\rvert\leq \lvert\partial  WW_m\rvert$ if $n\leq m$.  As every Wang--Wang set is minimal, $\lvert\partial WW_{\lvert A\rvert}\rvert= \lvert\partial  A\rvert$ and $\lvert\partial WW_{\lvert B\rvert}\rvert\leq  \lvert\partial  B\rvert$; thus $\lvert \partial A\rvert \leq \lvert \partial B\rvert$. 
Now suppose that $\lvert \partial A\rvert = \lvert \partial B\rvert$. If for some $C\subset X^0$ we have $\lvert B\rvert =\lvert C\rvert$, then as  $\lvert A\rvert \leq\lvert C \rvert$, we have $\lvert \partial B\rvert=\lvert \partial A\rvert\leq \lvert \partial C\rvert$. Thus $B$ is minimal.
\end{proof}

\section{Saturated sets}\label{sec:saturated}

\emph{We define the notion of a  saturated set, a  subset of $X^0$ with  the property that if any additional vertex is added to this set, then its boundary must increase.}
\vspace{.3cm}

\begin{definition}
A set $A\subset X^0$ is \emph{saturated} if 
  $|\partial(A\cup \{v\})|> |\partial A|$  for all $v\in X^0\setminus A$. 
\end{definition}

A \emph{configuration} is a subset $(F, N)\subset X^0\times X^0$ such that $F\cap N=\emptyset$. We say that two configurations $(F,N)$ and $(F',N')$ are \emph{congruent} if there is 
an automorphism $\phi \in \text{Aut}(X)$ 
such that $(F',N')=(\phi(F),\phi(N))$. A set $A\subset X^0$  \emph{contains the configuration} $(F, N)$ if  $F\subset A$ and $N\cap A=\emptyset$. Some configurations are shown in Figure~\ref{fig_forbidden}.

\begin{figure}[hbtp]
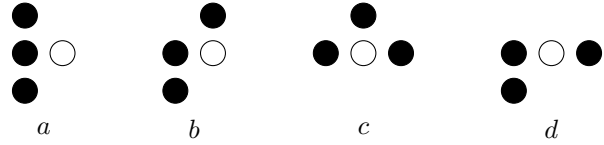
	
	\ctikzfig{images/configurations}
	\caption{Four configurations $(F,N)$ are shown, where elements of $F$ are shown as black vertices and elements of $N$ are shown as white vertices. Lemma~\ref{lemma_forbidden} ensures that a saturated set cannot contain a configuration congruent to one of the ones above.}
	\label{fig_forbidden}
\end{figure}

\begin{lemma} \label{lemma_forbidden}
	If $A\subset X^0$ is  saturated, then $A$ does not contain a configuration 
	congruent to one shown  in Figure~\ref{fig_forbidden}.
\end{lemma}
\begin{proof}
	Suppose $A$ contains a configuration $(F,N)$ from Figure~\ref{fig_forbidden}, and let $v \in N$. By definition $v \notin A$. Furthermore, it is straightforward to check that $\partial (A\cup \{v\})\leq \partial A$,  contradicting our assumption that  $A$ is saturated.
\end{proof}

\section{Excess}\label{sec:Excess}

\emph{We introduce excess,  a number associated to a subset $A \subset X^0$ that we will use in
later sections to characterize minimal sets and to study the structure of the
graph of minimal sets. }
\vspace{.3cm}

\begin{definition}\label{def:excess}
The \emph{excess} of $A\subset X^0$ is defined to be 
\[\Exc(A)\coloneqq \max \bigl\{ \, |A| - |B|  \bigm|  B \subset X^0 \text{ is
minimal and } \lvert\partial A\rvert=\lvert\partial B\rvert  \, \bigr\}.\]
\end{definition}

The following lemma shows that the excess of $A$ is  well-defined.

\begin{lemma} \label{lem:excess_well_defined}
For any finite $A\subset X^0$, there exists a minimal set $B\subset X^0$ such that $\lvert \partial A\rvert=\lvert \partial B\rvert$. 
\end{lemma}
\begin{proof}
 We first claim that if $n\geq 6$, then
 there exists a minimal set $B$ with $\left| \partial B\right|  = n$.
 Indeed, as noted in the proof of Lemma~\ref{lem:bdry monotone},  for every $m \ge 2$ either $|\partial WW_{m+1}|
 = |\partial WW_m|$ or $|\partial WW_{m+1}|
 = |\partial WW_m| + 1$. Moreover,  $\{ \left| \partial WW_{m}\right|\mid m\in \bN\} $ is unbounded.   Since $\left|\partial WW_2\right| =6$, the claim follows.
 
 If $\left| A\right| =1$, then $A$ is minimal. Otherwise,  $\lvert A\rvert \geq 2=\lvert  WW_2\rvert$, so Lemma~\ref{lem:bdry monotone} ensures that $| \partial A | \ge \left| WW_2\right|= 6$. 
By the preceding claim, there exists a minimal set $B$ with $|\partial A| = |\partial B|$.
\end{proof}

  The next two lemmas follow almost immediately from the definition of excess.

\begin{lemma}\label{lem:+ve exc minimal}
 A finite set $A \subset X^0$ is minimal if and only if $\Exc(A) \geq 0$. 
\end{lemma}
\begin{proof}
 If $A$ is minimal, then by taking $B = A$ in the definition of excess, we
 get that~$\Exc (A) \ge 0$. 
 For the converse, if $\Exc (A) \ge 0$, then there
 exists a minimal set $B$ such that $|\partial A| = |\partial B|$ and $|A|
 - |B| \ge 0$. Thus, $A$ is minimal by Lemma~\ref{lem:bdry monotone}.
\end{proof}

\begin{lemma} \label{lemma_def_same_boundary}
 If $A,A' \subset X^0$ are finite and  $|\partial A| = |\partial A'|$, then  $\Exc(A) - \Exc(A') = |A| - |A'|$.
\end{lemma}
\begin{proof}
Since $|\partial A| = |\partial A'|$, there exists a minimal set $B\subset X^0$ such that $\Exc(A)=\left| A\right| -\left| B\right| $ and $\Exc(A')=\left| A'\right| -\left| B\right| $.  Thus, 
 \[\Exc(A) - \Exc(A') = \bigl(|A| - |B|  \bigr) - \bigl(|A'| - |B|  \bigr) = |A|- |A'|. \qedhere\]
\end{proof}

\section{Boxes} \label{sec:box_excess}
\emph{
In this section, we define boxes. These are sets that are bounded by lines of slope $1$ and $-1$. We prove some key facts regarding these sets and give an explicit formula for their excess. 
As a consequence, we determine which boxes are minimal sets. 
}
\vspace{.3cm}

Proposition~\ref{prop:charbox} demonstrates that the following definition of a box is  equivalent to the one given in the introduction.
\begin{definition} \label{def:box}
 A \emph{box} in $X$ is a nonempty
  subset of $X^0$ of the form \[B(a,b,c,d)\coloneqq \{(x,y)\in X^0\mid a\leqslant y-x\leqslant b \text{ and } c\leqslant y + x \leqslant d\}\]  
 for some $a, b, c, d\in \mathbb{Z}$.
\end{definition}
 
\begin{conv}\label{conv:box}
 There is an ambiguity when giving coordinates for boxes, which only arises for degenerate boxes of the form $B(a,b,c,d)$ where either $a=b$ or $c=d$. For example, $B(0,0,0,3)$ and $B(0,0,0,2)$ are the same box. To remedy this issue and ensure boxes can be uniquely parametrized, given a box $B(a, b, c, d)$ we implicitly assume that $a$ and $c$ are maximal and that $b$ and $d$ are minimal out of all possible choices.
\end{conv}
Given a subset $A\subset X^0$, we define $N(A)\coloneqq A\sqcup \partial A$.
 We next show that boxes are saturated. In the next section, we prove a converse 
 to this statement for minimal sets
  (see Proposition~\ref{prop:finminsat -> box}). 

\begin{lemma}\label{lem:box_is_sat}
Every box is saturated.
\end{lemma}
\begin{proof}
Let $B=B(a,b,c,d)$ be a box. 
Suppose  $v=(x, y)\in X^0\setminus B$.
As $v \notin B$,
$x$ and $y$ do not
satisfy one of the four defining equations of $B$. 
Without loss of generality, we  assume $x+y>d$. 
It follows that $(x+1,y)$ and $(x,y+1)$ are adjacent to $v$ and not contained in 
$N(B)$. 
As $\lvert N (B\cup\{v\})\rvert \geq \lvert N( B)\rvert+2$, we have  
\[\lvert \partial (B\cup\{v\})\rvert=\lvert N(B\cup\{v\})\rvert - \lvert B\cup\{v\}\rvert  \geq( \lvert N(B)\rvert+2) - (\lvert B\rvert+1)= \lvert\partial B\rvert +1>\lvert\partial B\rvert,\]
demonstrating that $B$ is saturated.
\end{proof}

An \emph{extremal line} of a box $B(a,b,c,d)$  is the set of solutions in $\Z ^2$  to  one of the four equations $y-x=a$, $y-x=b$, $y+x=c$ or $y+x=d$. By the above convention, a box intersects each of its extremal lines in at least one point. A \emph{corner} of the box $B$ is an element of $B$ that lies on the intersection of two distinct extremal lines. A box has either zero, two, or four corners. An example of a box with two corners is shown in the left of Figure~\ref{fig:boxes} and one with no corners is shown on the right of that figure.

\begin{figure}[htp]
	\begin{subfigure}[B]{0.4\textwidth}
		\scalebox{1}{\begin{tikzpicture}
	\begin{pgfonlayer}{nodelayer}
		\node [style=none] (16) at (1.25, -3.25) {\tiny $y\!+\!x\!=\!0$};
		\node [style=none] (17) at (6, 1) {\tiny  $y\!+\!x\!=\!9$};
		\node [style=none] (18) at (6.25, 4.25) {\tiny $y\!-\!x\!=\!0$};
		\node [style=none] (19) at (3.75, 5.75) {\tiny $y\!-\!x\!=\!4$};
		\node [style=none] (21) at (1.5, 5.5) {};
		\node [style=none] (22) at (5.75, 1.25) {};
		\node [style=none] (23) at (-3, 1) {};
		\node [style=none] (24) at (1, -3) {};
		\node [style=none] (25) at (-3, -1) {};
		\node [style=none] (26) at (3.5, 5.5) {};
		\node [style=none] (27) at (-1, -3) {};
		\node [style=none] (28) at (6, 4) {};
		\node [style=Black] (29) at (-2, 0) {};
		\node [style=Black] (30) at (-1, 1) {};
		\node [style=Black] (31) at (-1, 0) {};
		\node [style=Black] (32) at (-1, -1) {};
		\node [style=Black] (33) at (0, 1) {};
		\node [style=none] (34) at (0, -2) {};
		\node [style=Black] (35) at (0, -1) {};
		\node [style=Black] (36) at (0, -2) {};
		\node [style=Black] (37) at (0, 2) {};
		\node [style=Black] (38) at (1, 3) {};
		\node [style=Black] (39) at (1, 2) {};
		\node [style=Black] (40) at (1, 1) {};
		\node [style=Black] (41) at (1, 0) {};
		\node [style=Black] (42) at (1, -1) {};
		\node [style=Black] (43) at (2, 4) {};
		\node [style=Black] (44) at (2, 3) {};
		\node [style=Black] (45) at (2, 2) {};
		\node [style=Black] (46) at (2, 1) {};
		\node [style=Black] (47) at (2, 0) {};
		\node [style=Black] (48) at (3, 4) {};
		\node [style=Black] (49) at (3, 3) {};
		\node [style=Black] (50) at (3, 2) {};
		\node [style=Black] (51) at (3, 1) {};
		\node [style=Black] (52) at (4, 3) {};
		\node [style=Black] (53) at (4, 2) {};
		\node [style=none] (54) at (0, 5.75) {};
		\node [style=none] (55) at (6, -2) {};
		\node [style=none] (56) at (0, -3) {};
		\node [style=none] (57) at (-3, -2) {};
		\node [style=none] (58) at (5.75, -2.25) {\tiny $x$};
		\node [style=none] (59) at (-0.25, 5.5) {\tiny $y$};
		\node [style=Black] (60) at (0, 0) {};
	\end{pgfonlayer}
	\begin{pgfonlayer}{edgelayer}
		\draw [style=none] (21.center) to (22.center);
		\draw [style=none] (23.center) to (24.center);
		\draw [style=none] (25.center) to (26.center);
		\draw [style=none] (27.center) to (28.center);
		\draw [style=Black arrow] (34.center) to (55.center);
		\draw [style=Black arrow] (34.center) to (54.center);
		\draw [style=none] (34.center) to (56.center);
		\draw [style=none] (34.center) to (57.center);
	\end{pgfonlayer}
\end{tikzpicture} }
		\caption{The box $B(0,4,0,9)=B(4,9)$.}
		\label{fig:box}
	\end{subfigure}
	\hfill
	\begin{subfigure}[B]{0.4\textwidth}
		\vspace{.7cm}
		\scalebox{1}{\begin{tikzpicture}
	\begin{pgfonlayer}{nodelayer}
		\node [style=none] (16) at (0.25, -3.25) {\tiny $y\!+\!x\!=\!-1$};
		\node [style=none] (17) at (4, -1) {\tiny  $y\!+\!x\!=\!5$};
		\node [style=none] (18) at (4, 2) {\tiny $y\!-\!x\!=\!0$};
		\node [style=none] (19) at (2, 4) {\tiny $y\!-\!x\!=\!4$};
		\node [style=none] (21) at (-1, 4) {};
		\node [style=none] (22) at (3.75, -0.75) {};
		\node [style=none] (23) at (-3.75, 0.75) {};
		\node [style=none] (24) at (0, -3) {};
		\node [style=none] (25) at (-3.5, -1.5) {};
		\node [style=none] (26) at (1.75, 3.75) {};
		\node [style=none] (27) at (-1, -3) {};
		\node [style=none] (28) at (3.75, 1.75) {};
		\node [style=Black] (29) at (-2, 0) {};
		\node [style=Black] (30) at (-1, 1) {};
		\node [style=Black] (31) at (-1, 0) {};
		\node [style=Black] (32) at (-1, -1) {};
		\node [style=Black] (33) at (0, 1) {};
		\node [style=none] (34) at (0, -2) {};
		\node [style=Black] (35) at (0, -1) {};
		\node [style=Black] (36) at (0, -2) {};
		\node [style=Black] (37) at (0, 2) {};
		\node [style=Black] (39) at (1, 1) {};
		\node [style=Black] (40) at (1, 0) {};
		\node [style=Black] (41) at (1, -1) {};
		\node [style=Black] (42) at (-1, -2) {};
		\node [style=Black] (44) at (-2, -1) {};
		\node [style=Black] (46) at (1, 2) {};
		\node [style=Black] (47) at (2, 0) {};
		\node [style=Black] (53) at (2, 1) {};
		\node [style=none] (54) at (0, 3.75) {};
		\node [style=none] (55) at (4, -2) {};
		\node [style=none] (57) at (-3, -2) {};
		\node [style=none] (58) at (4, -2.25) {\tiny $x$};
		\node [style=none] (59) at (-0.25, 3.5) {\tiny $y$};
		\node [style=Black] (60) at (0, 0) {};
		\node [style=none] (61) at (-4.25, -3.25) {};
	\end{pgfonlayer}
	\begin{pgfonlayer}{edgelayer}
		\draw [style=none] (21.center) to (22.center);
		\draw [style=none] (23.center) to (24.center);
		\draw [style=none] (25.center) to (26.center);
		\draw [style=none] (27.center) to (28.center);
		\draw [style=Black arrow] (34.center) to (55.center);
		\draw [style=Black arrow] (34.center) to (54.center);
		\draw [style=none] (34.center) to (57.center);
	\end{pgfonlayer}
\end{tikzpicture} }
		\caption[(b)]{The box $B(0,4,-1,5)=\hat B(4,6)$.}
		\label{fig:box2}
	\end{subfigure}
\caption{}
	\label{fig:boxes}
\end{figure}

	

	
\begin{definition}
	The {\it modulus} of a box $B(a,b,c,d)$ is the unordered pair $\{b-a,d-c\}$. 
\end{definition}

When Convention~\ref{conv:box} is followed, it is evident that the modulus of a box is well-defined and is invariant under congruence. 
We intuitively expect that a box of modulus $\{\alpha,\beta\}$ is minimal when $|\alpha - \beta|$ is small. We precisely quantify how small $|\alpha - \beta|$ must be in  Theorem~\ref{thm:box_excess} and Remark~\ref{rem:min boxes}.

We now show that a box of modulus $\{\alpha, \beta\}$ is congruent to a ``standard box'' $B(\alpha, \beta)$ or $\hat{B}(\alpha, \beta)$ as defined below.

\begin{notation} \label{notation:boxes}
 Define $B(\alpha,\beta)\coloneqq B(0,\alpha,0,\beta)$ for any $\alpha,\beta \in \bN$,  and $\hat B(\alpha,\beta)\coloneqq B(0,\alpha,-1,\beta-1)$ for any even $\alpha,\beta \in \bN$ (see Figure~\ref{fig:boxes} for examples).
\end{notation}


\begin{prop}\label{prop:charbox}
	Let $B = B(a,b,c,d)$ be a box. If $B$ has no corners, then $b-a$ and $d-c$ are both even and $B$ is congruent to $\hat B(b-a,d-c)$. Otherwise, $B$ is congruent  to $B(b-a,d-c)$.
\end{prop}
\begin{proof}
	Suppose first that $B$ contains no corners. Then the  intersection of the line $y - x = a$ with the line $y+x =  c$ does not have integer coordinates, so  $a$ and $c$ have 
	opposite parity. Similarly, we deduce that  $a$ and $d$ have opposite parity and that $b$ and $c$ have opposite parity. Thus $b-a$ and  $d-c$ must both be  even. 
	Let $u$ be the vertex of $B$ which lies on the the line $y = x + a$ and has minimal $y$-value of all such possible choices. We can apply a translation which sends $u$ to the origin $(0,0)$. The resulting box is then $B(0,b-a,-1,d-c-1) = \hat B (b-a, d-c)$ as desired.
	
	On the other hand, suppose that $B$ contains a corner $v$. Then there exists an automorphism of $X$ sending $v$ to the origin that maps $B$ to the box   $B(b-a,d-c)$.
\end{proof}

The next two lemmas calculate the size of a box and its boundary. 
\begin{lemma}\label{lem:bdry}
Let $B$ be a box with modulus $\{\alpha,\beta\}$. Then $\lvert \partial B\rvert =\alpha+\beta+4$.
\end{lemma}
\begin{proof} 
Let  $B = B(a, b, c, d)$ be a box of modulus $\{\alpha = b-a, \beta = d-c \}$. We can assume without loss of generality that $\alpha\leq \beta$. First suppose that $\alpha=0$, in which case $B$ is congruent to $B(0,\beta)$. Note that Convention~\ref{conv:box} implies $\beta$ must be even. Since $\lvert \partial B(0,0)\rvert=4$ and $\lvert \partial B(0,2n+2)\rvert =\lvert \partial B(0,2n)\rvert +2$ for every $n\in \bN$, it follows by induction that $\lvert \partial B(0,\beta)\rvert =\beta+4$ for all even $\beta$.

We thus assume that $\alpha\geq 1$. If $\alpha=\beta=1$, then $B$ is congruent to $B(1,1)$ and the formula $\lvert \partial B\rvert =\alpha+\beta+4=6$ clearly holds. We thus also assume that $\beta\geq 2$ and proceed by induction on $\beta$. We assume that the lemma is true for all boxes of modulus $\{\alpha',\beta'\}$, where $\max(\alpha',\beta')<\beta$.
Let $L$ and $L^+$ be the lines with equations $y = -x + d$ and $y = -x + d + 1$ respectively.  Let $V := \partial (L \cap B) \cap L^+$ and observe that  $|V| =  |L \cap B| + 1$.
Let $B'= B(a, b, c, d-1)$. Since  $d-c \ge 2$ and $b-a \ge 1$, the preceding parametrization of $B'$ is consistent with Convention~\ref{conv:box}. Thus $B'$ has modulus $\{\alpha,\beta-1\}$. We observe that $N(B)=N(B')\sqcup V$ and $B=B'\sqcup (B\cap L)$. The claim now follows from the equation below, where the last equality uses our induction hypothesis.
\[\left| \partial B\right| =\left| N(B)\right| -\left| B\right| =(\left| N(B')\right|+\left| V\right|)  -(\left| B'\right|+\left| L\cap B\right| )=\left| \partial B'\right| +1=\alpha+\beta+4. \qedhere\]
\end{proof}

\begin{lemma}\label{lem:area}
Let $\alpha,\beta\in \bN$, we have that 	\[\lvert B(\alpha,\beta)\rvert =\left\lfloor\frac{\alpha\beta+\alpha+\beta+2}{2}\right\rfloor.\]
Moreover, if $\alpha$ and $\beta$ are both even, then 
	\[\lvert \hat B(\alpha,\beta)\rvert =\frac{\alpha\beta+\alpha+\beta}{2}.\]
\end{lemma}
\begin{proof}
Let $p:X^0\rightarrow \Z$ be the projection map given by $(x,y)\mapsto y-x$.
Let $B$ be a box and let $I$ be the interval $p(B)$. Thus $\lvert B\rvert =\sum_{i\in I}\lvert p^{-1}(i)\cap B\rvert$. 
We break the proof into cases depending on the type of box $B$ and the parity of $\alpha$ and $\beta$.

We first analyze the case where $B = B(\alpha, \beta)$. Suppose $\beta$ is even. It follows that $|p^{-1}(i)| = \frac{\beta}{2}+1$ for even $i \in I$ and $|p^{-1}(i)| = \frac{\beta}{2}$
 for odd $i \in I$. Thus, if $\alpha$ is even, then 
\[\lvert B \rvert = \left( \frac{\alpha}{2} + 1 \right) \left( \frac{\beta}{2}+1 \right) + \left(\frac{\alpha}{2}  \frac{\beta}{2} \right)= \frac{\alpha\beta+\alpha+\beta+2}{2} = \left\lfloor\frac{\alpha\beta+\alpha+\beta+2}{2}\right\rfloor. \]
If $\alpha$ is odd, we have:
\[\lvert B \rvert = \left(\frac{\alpha + 1}{2} \right) \left(\frac{\beta}{2}+1 \right) + \left( \frac{\alpha + 1}{2} \right) \left( \frac{\beta}{2} \right) = \frac{\alpha\beta+ \alpha + \beta+1}{2} = \left\lfloor\frac{\alpha\beta+\alpha+\beta+2}{2}\right\rfloor.  \]
The last equality follows since $\frac{\alpha\beta+ \alpha + \beta+1}{2}$ is equal to $|B|$ and hence it is an integer.

On the other hand, suppose that $\beta$ is odd. In this case, it follows that $|p^{-1}(i)| = \frac{\beta + 1}{2}$ for all $i \in I$. Thus, 
\[\lvert B \rvert = \sum_{j = 0}^{\alpha} \frac{\beta + 1}{2} = (\alpha+1) \left(\frac{\beta + 1}{2}\right) = \frac{\alpha\beta+ \alpha + \beta+1}{2} = \left\lfloor\frac{\alpha\beta+\alpha+\beta+2}{2}\right\rfloor . \]

Finally, let $B = \hat B(\alpha, \beta)$ where both $\alpha$ and $\beta$ are even. It follows that $|p^{-1}(i)| = \frac{\beta}{2}$ for even $i \in I$ and $|p^{-1}(i)| = \frac{\beta}{2} + 1$ for odd $i \in I$. Thus, 
\[\lvert B \rvert = \left(\frac{\alpha}{2} + 1 \right) \frac{\beta}{2} + \frac{\alpha}{2} \left( \frac{\beta}{2} + 1 \right) = \frac{\alpha\beta + \alpha + \beta}{2}.\qedhere\]
\end{proof}

\begin{remark} \label{rmk:B vs B hat}
If $\alpha,\beta \in \N$ are both even, then by Lemma~\ref{lem:bdry} and Lemma~\ref{lem:area} we have $\lvert \partial \hat B(\alpha,\beta)\rvert=\lvert\partial  B(\alpha,\beta)\rvert$ and $\lvert \hat B(\alpha,\beta)\rvert+1=\lvert B(\alpha,\beta)\rvert$.
\end{remark}

The following lemma allows us to take a nested sequence of subsets of a box, all with the same size boundary. This lemma will be used in Sections~\ref{sec:connected} and~\ref{sec_finite_comps} as well as here. 

 \begin{lemma} \label{lemma_exc_bound}
	Let $B$ be a box of modulus $\{ \alpha, \beta \}$, and let $L$ be an extremal line of $B$. Set  $n = |L \cap B| - 1$.  Then the following are true:
	\begin{enumerate}
		\item If $\alpha, \beta \ge 2$, then there exist sets $B = B_0 \supset B_1 \supset \ldots \supset B_{n}$ such that $|\partial B_i| = |\partial B_{i-1}|$ for all $1 \le i \leq n$.
		\item $\Exc(B) \le n$
		\item If $L'$ is a standard line that intersects $B$, then $\Exc(B)\leq \lvert L'\cap B\rvert$.
	\end{enumerate}
\end{lemma}

\begin{proof}
	Without loss of generality, we may assume that $B = B(a, b, c, d)$ and that $L$ is the line with equation $y = -x + d$. 
	
	We first suppose that $\alpha, \beta \ge 2$, and we prove claims (1) and (2) in this case.
	Let $v_1 = (x_1, y_1), \dots, v_{n+1} = (x_{n+1}, y_{n+1})$ be the vertices of $L \cap B$, ordered so that $x_1<x_2<\ldots< x_{n+1}$. 
	Let $B_i := B \setminus \{v_1, \dots, v_i \}$. 
	It follows from our hypothesis on the modulus of $B$ that, for each $1\leq i\leq n$, the vertices $(x_i, y_i)$, $(x_{i}- 1,y_{i})$, $(x_{i} + 1,y_{i})$  and $(x_{i},y_{i}-1)$ are each contained in $N(B_{{i}})$ and $(x_i, y_i + 1)$ is not. Thus  $N(B_{i-1})=N(B_{i})\sqcup \{(x_{i},y_{i}+1)\}$ for each $1\leq i\leq n$, and so $\lvert \partial B_i\rvert=\lvert \partial  B_{i-1}\rvert$. This shows (1).
	To see (2), 
	note that $B_{n+1}$ is a box whose modulus is either $\{\alpha-1,\beta\}$ or $\{\alpha,\beta-1\}$. By Lemma~\ref{lem:bdry}, $|\partial B_{n+1} | = |\partial B| - 1$. 
	Thus, any minimal set of size $|B_{n+1}| = |B| - (n+1)$ must have boundary of size at most $|\partial B| - 1$. It follows that
	$\Exc(B) \le n$. Thus, (2) follows in this case. 
	
	We now suppose that $\beta := d-c \ge \alpha :=b-a$ and that $\alpha \le 1$, and we prove  (2) for this remaining case. 
	As before, let $L$ be the line with equation $y = -x + d$.
	Since $\alpha\leq 1$, $L$ intersects $B$ in a single vertex $v$. 
	It follows that $B' := B \setminus v$ is a box and is of strictly smaller modulus.
	Thus Lemma \ref{lem:bdry} implies  that  $\lvert \partial B'\rvert < \lvert \partial B\rvert$. It follows from Lemma \ref{lem:bdry monotone} that any minimal set of size at most $\lvert B\rvert-1$ has boundary of size at most $\lvert \partial B\rvert-1$. This implies $\Exc(B)\leq 0$ as required. If $Q$ is any other extremal line of $B$, then $|Q \cap B| \ge |L \cap B| = 1$ and we also get that $\Exc(B) \le |Q \cap B| - 1$. Thus, (2) follows.
	
	Finally, to see (3), suppose $L'$ is a standard line that intersects $B$, and let $L''$ be the extremal line of $B$ that is  parallel to $L'$. It follows that $\lvert L'' \cap B\rvert \leq \lvert L' \cap B\rvert +1$. By what we have shown, we get that $\Exc(B)\leq \lvert L'' \cap B\rvert -1$ as required.
\end{proof}

\begin{remark}
	By considering the box $B=B(2,2)$, which has excess one, we see that the bounds for $\Exc(B)$  given in the previous lemma are sharp.
\end{remark}

To calculate the excess of an arbitrary box, we first compute the excess of a box with modulus $\{\alpha,\beta\}$, where $\lvert \alpha-\beta \rvert\leq  1$.

\begin{lemma}\label{lem:special box1 excess}
For every $n\in \N$, 
\begin{itemize}
\item  $\Exc\bigl(B(2n,2n)\bigr)=n$;
\item  $\Exc\bigl(B(2n+2, 2n+3)\bigr)=n$;
\item  $\Exc\bigl(B(2n+1,2n+1)\bigr)=n$;
\item $\Exc\bigl(B(2n+1,2n+2)\bigr)=n$;
\item $\Exc\bigl(\hat B(2n,2n)\bigr)=n-1$.
\end{itemize}
\end{lemma}
\begin{proof}
We first remark that if $A\subset X^0$ is minimal, then $\lvert\partial A\rvert=\lvert\partial WW_{\lvert A\rvert}\rvert$. Thus for any minimal set $A$ of $X$, $\Exc(A)= \max\{k\mid \lvert \partial WW_{\lvert A\rvert-k}\rvert=\lvert \partial WW_{\lvert A\rvert}\rvert\}$. 

Let $Y$ be one of $B(2n,2n)$, $B(2n +2, 2n+3)$ 
, $B(2n+1,2n+1)$ or $B(2n+1,2n+2)$. Then $Y$ is congruent to a Wang--Wang set $WW_m$ for some $m$. It can be verified by Lemma~\ref{lemma_exc_bound}(1)
 that $\lvert \partial WW_{m-i}\rvert = \lvert \partial WW_{m}\rvert$ if and only if $i\leq n$. 
This gives the required formula for the excess of $Y$. Finally, by Lemma~\ref{lemma_def_same_boundary} and Remark~\ref{rmk:B vs B hat}, $\Exc\bigl(\hat B(2n,2n)\bigr)=\Exc\bigl(B(2n,2n)\bigr)-1=n-1$.
\end{proof}

We are now ready to calculate the excess of any box.

\begin{thm} \label{thm:box_excess}
Suppose $\alpha,\beta\in \N$. Let $r\coloneqq\frac{\alpha+\beta}{2}$, $k\coloneqq\frac{\lvert \beta-\alpha\rvert}{2}$.  Then
\[\Exc\bigl(B(\alpha,\beta)\bigr)=\left \lfloor \frac{\lfloor r\rfloor -k^2}{2}\right\rfloor. \]
Moreover, when $\alpha,\beta\in \N$ are even, then 
\[\Exc\bigl(\hat B(\alpha,\beta)\bigr)=\frac{r-k^2-2}{2}.\]
\end{thm}
\begin{proof}
By applying an automorphism of $X$, we may assume that $\beta \geq \alpha$. Note that $\alpha = r-k$ and $\beta = r+k$. We break the argument into two cases. 
\\
\textit{Case A: $\alpha$ and $\beta$ have the same parity.}

In this case, $r$ and $k$ are both integers.
Lemma~\ref{lem:area} yields the equations: 
\[\lvert B(r,r)\rvert = \left\lfloor\frac{r^2+2r+2}{2}\right\rfloor=\left\lfloor\frac{r^2}{2}\right\rfloor+r+1, \]
\[\lvert B(r - k,r + k)\rvert = \left\lfloor\frac{r^2 - k^2 + 2r + 2}{2}\right\rfloor=\left\lfloor\frac{r^2 - k^2}{2}\right\rfloor+r+1. \] 
Setting $r=2m+\epsilon$ where $\epsilon\in \{0,1\}$ and $m\in \bZ$, we get the equation:
\[\left\lfloor\frac{r^2 - k^2}{2}\right\rfloor-\left\lfloor\frac{r^2}{2}\right\rfloor=
\left\lfloor\frac{\epsilon^2 - k^2}{2}+2m^2+2m\epsilon\right\rfloor-\left\lfloor\frac{\epsilon^2}{2}+2m^2+2m\epsilon\right\rfloor=\left\lfloor\frac{\epsilon - k^2}{2}\right\rfloor\]
It follows from Lemma~\ref{lem:bdry} that $\lvert \partial B(r,r)\rvert=\lvert \partial B(r-k,r+k)\rvert$. Lemma~\ref{lemma_def_same_boundary} then implies
\begin{align*}
\Exc(B(\alpha,\beta)) &= \Exc\bigl(B(r-k,r+k)\bigr)=\Exc\bigl(B(r,r)\bigr)-\lvert B(r,r)\rvert +\lvert B(r-k,r+k)\rvert\\
&=\Exc\bigl(B(r,r)\bigr)+\left\lfloor\frac{r^2 - k^2}{2}\right\rfloor-\left\lfloor\frac{r^2}{2}\right\rfloor=\left\lfloor\Exc\bigl(B(r,r)\bigr)+\frac{\epsilon - k^2}{2}\right\rfloor.
\end{align*}
When  $r$ is even (and so $\epsilon=0$), $\Exc\bigl( B(r,r) \bigr) = \frac{r}{2}$ by Lemma~\ref{lem:special box1 excess}. Substituting this  into the above equation yields
	\begin{align*}
		\Exc\bigl(B(\alpha,\beta)\bigr) =  \left\lfloor \frac{r - k^2}{2} \right\rfloor = \left\lfloor \frac{ \lfloor r \rfloor - k^2}{2} \right\rfloor.
	\end{align*}
When  $r$ is odd (and so $\epsilon=1$), $\Exc\bigl( B(r,r) \bigr) = \frac{r - 1}{2}$ by Lemma~\ref{lem:special box1 excess}. Thus, 
	\begin{align*}
		\Exc\bigl(B(\alpha,\beta)\bigr) &=\left\lfloor \frac{r - 1}{2} +\frac{1 - k^2}{2}\right\rfloor =\left\lfloor\frac{\lfloor r\rfloor-k^2}{2}\right\rfloor.
	\end{align*}
\\
\textit{Case B: $\alpha$ and $\beta$ have different parity.}
We can write $r=s + \frac{1}{2}$ and $k=t + \frac{1}{2}$, for some $s,t\in \bZ$. Note that $\lfloor r\rfloor =s$, so we need to show 
\begin{equation}
\Exc\bigl(B(\alpha,\beta)\bigr)=\left \lfloor \frac{s -t^2-t-\frac{1}{4}}{2}\right\rfloor=\left \lfloor \frac{s }{2}-\frac{1}{8}\right\rfloor-\frac{t^2+t}{2} 
 \label{eq:need_to_show}
\end{equation}
	The last equality follows since $t^2+t$ is even.
	Lemma~\ref{lem:area} now yields
\begin{eqnarray*}
 \left\lvert B\left(s,s+1\right)\right\rvert &=& \left\lfloor \frac{s^2+3s+3}{2}\right\rfloor \\
 \lvert B(\alpha, \beta) \rvert = \lvert B(s-t,s+t+1)\rvert &=& \left\lfloor\frac{s^2+3s+3-t^2-t}{2}\right\rfloor=\left\lvert B\left(s,s+1\right)\right\rvert-\frac{t^2+t}{2},
\end{eqnarray*}
where  the last equality follows again because $t^2+t$ is even.

As in Case~A,
\begin{eqnarray*}
 \Exc\bigl(B(s-t, s+t+1)\bigr)  &=& \Exc\left(B\left(s,s+1\right)\right)-\left\lvert B\left(s,s+1\right)\right\rvert +\left\lvert B(s-t,s+t+1)\right\rvert\\
 &=& \Exc\left(B\left(s,s+1\right)\right) -\frac{t^2+t}{2},
\end{eqnarray*}
 So by Equation~\ref{eq:need_to_show}, 
 we need only to verify that $\Exc\left(B\left(s,s+1\right)\right)=\left \lfloor \frac{s }{2}-\frac{1}{8}\right\rfloor$, or equivalently: 
 \[\left \lfloor \frac{s }{2}-\frac{1}{8}-\Exc\left(B\left(s,s+1\right)\right)\right\rfloor=0.\] 
Lemma~\ref{lem:special box1 excess} ensures $\Exc(B(s, s+1)) = \frac{s}{2} - 1$ when $s$ is even, and   $\Exc(B(s, s+1)) = \frac{s}{2}-\frac{1}{2}$ when $s$ is odd, so the preceding equation is satisfied in both cases.

Finally, Lemma~\ref{lemma_def_same_boundary} and Remark~\ref{rmk:B vs B hat} imply that  $\Exc\bigl(\hat B(\alpha,\beta)\bigr)=\Exc\bigl(B(\alpha,\beta)\bigr) - 1 =\frac{r-k^2-2}{2}$  for even $\alpha,\beta\in \bN$.
\end{proof}

\begin{remark}\label{rem:min boxes}
	By combining Lemma~\ref{lem:+ve exc minimal} and Theorem~\ref{thm:box_excess}, we have a  complete characterization of which boxes are minimal. 
\end{remark}

\section{Characterizing minimal sets}\label{sec:characterization}
\emph{In this section we prove Theorem~\ref{thm:minimal set characterization}, giving our first characterization of  minimal sets in the graph $X = \Z^2_{\ell_1}$. We also prove Proposition~\ref{prop:sat+conn->box}, which characterizes boxes as precisely the  sets that  are both saturated and $\ell_\infty$-connected. 
}
\vspace{.3cm}

We first explain how to deduce Theorem \ref{thm:minimal set characterization} from the following proposition, whose proof occupies the remainder of this section.
\begin{prop}\label{prop:finminsat -> box}
	If $A\subset X$ is  minimal and saturated, then it is a box.
\end{prop}

\begin{definition}
	Given a finite set $A\subset X^0$, the \emph{enclosing box}, denoted $\enc(A)$, is the smallest box that contains $A$.
\end{definition}
The enclosing box of a set is well-defined, as the intersection of boxes is itself a box.
The enclosing box of a minimal set is the unique smallest saturated set containing it:

\begin{lemma} \label{lem:nested_sequence_to_enc_box}
	Let $A\subset X^0$ be a minimal set, and let $A = A_0 \subset A_1 \subset A_2 \subset \dots $ be a (possibly finite) maximal sequence of nested minimal sets such that $|A_{i+1}| = |A_i|+1$  for each $i$. Then $A_N=\enc(A)$ for some $N \ge 0$ and $|\partial A| = |\partial (\enc(A))|$.
\end{lemma}
\begin{proof}
	Let $N \ge 0$ be the largest integer such that $|\partial A_N| = |\partial A|$. 
	Such an integer exists as there are minimal sets with arbitrarily large boundaries (see the proof of Lemma~\ref{lem:excess_well_defined} for instance).
	To prove the lemma, it suffices to show that $A_N=\enc(A)$. 
	
	We first claim that $A_N$ must be saturated. For a contradiction, suppose otherwise. It follows there exists a set $A_N' \supset A_N$ such that $|\partial A_N'| = |\partial A_N|$ and $|A_N'| = |A_N|+1$.
	Furthermore, by Lemma \ref{lem:bdry monotone}, $A_N'$ is minimal.  
	By the maximality of our nested sequence, it contains a set $A_{N+1}$. 
	As $A_{N+1}$ is minimal and $|A_{N+1}| = |A_N'|$, it follows that $|\partial A_{N+1}| = |\partial A_N|$, contradicting our choice of $N$.
	
	As $A_N$ is saturated and minimal, it is a box by Proposition~\ref{prop:finminsat -> box}. If $\enc(A)\neq A_N$, then $\enc(A)$ must be a proper subset of $A_N$. However, in this case, we then have that $|\partial (\enc(A))| < |\partial A_N|$ by Lemma~\ref{lem:bdry}, contradicting our choice of $N$. Thus $A_N=\enc(A)$, and the lemma follows.
\end{proof}

We are ready to prove our first characterization of minimal sets, using Proposition \ref{prop:finminsat -> box}.
\begin{thm}\label{thm:minimal set characterization}
	Let $A\subset X^0$, $N\coloneqq \lvert \enc(A)\setminus A \rvert$ and $E=\Exc(\enc(A))$. Then $A$ is minimal if and only if $\lvert \partial A\rvert = \lvert \partial (\enc(A))\rvert$ and $E\geq N$.
\end{thm}
\begin{proof}
	First suppose $A$ is minimal. 
	By Lemma~\ref{lem:nested_sequence_to_enc_box}, $|\partial A| = |\partial \enc(A)|$. Thus by Definition~\ref{def:excess}, $E \ge N$. Conversely, suppose that $\lvert \partial A\rvert = \lvert \partial (\enc(A))\rvert$ and $E \geq N$. It follows from Definition~\ref{def:excess} that there exists a minimal set $C$ such that $\lvert \partial C \rvert =\lvert \partial (\enc(A))\rvert = \lvert \partial A\rvert$ and \[\lvert C\rvert=\lvert \enc(A)\rvert-E\leq \lvert \enc(A)\rvert-N=\lvert A\rvert.\]  
	By  Lemma~\ref{lem:bdry monotone}, $A$ is minimal.
\end{proof}

We now begin our proof of Proposition \ref{prop:finminsat -> box}. We first establish some terminology regarding two metrics on $X^0$: the $\ell_1$-metric and the $\ell_\infty$-metric.
We say that two vertices in $X^0$ are \emph{adjacent} if their distance is exactly $1$ in the $\ell_1$-metric.  An \emph{$\ell_1$-path} is a sequence $(u_i)_{i=0}^n$ of elements of $X^0$ such that $u_{i-1}$ and $u_{i}$ are adjacent for every $0 < i \le n$. We say  $A\subset X^0$ is \emph{connected} if any $x,y\in A$ are the endpoints of an $\ell_1$-path contained in $A$. 

We recall the {\it $\ell_\infty$-metric} (also known as the  Chebyshev, maximum, or chessboard metric) on~$X^0$, which is defined by \[d_\infty\bigl((x,y),(x',y')\bigr)\coloneqq\max\bigl(\lvert x-x'\rvert, \lvert y-y'\rvert\bigr).\] 
Two vertices  $u,v\in X^0$ are said to be \emph{$\ell_\infty$-adjacent} if $d_\infty(u,v)=1$.
An \emph{$\ell_\infty$-path} is a sequence $(u_i)_{i=0}^n$ of elements of $X^0$ such that $u_{i-1}$ and $u_{i}$ are $\ell_\infty$-adjacent for every $0 < i \le n$. 
 A subset $A\subset X^0$ is \emph{$\ell_\infty$-connected} if any pair of 
vertices in $A$  
  are the endpoints of an $\ell_\infty$-path contained in $A$. An  \emph{$\ell_\infty$-component} of $A$ is a maximal  $\ell_\infty$-connected subset of $A$.

A \emph{standard line} is a set of the form $\{(x,y)\mid y-x= w\}$ or
$\{(x,y)\mid y+x=w\}$ for some $w\in \bZ$. The next two lemmas will be needed to prove Proposition~\ref{prop:sat+conn->box}, our characterization of boxes.
\begin{lemma}\label{lem:thick box}
	Suppose $A\subset X^0$ is saturated and $C\subseteq A$ is an $\ell_\infty$-component of $A$ that is not contained in a standard line. Then $C$ contains a pair of adjacent vertices.
\end{lemma}
\begin{proof}
	As $C$ is $\ell_\infty$-connected and not contained in a standard line, it contains an $\ell_\infty$-path $\gamma=(u_0, \dots u_n)$  such that $u_0$ and $u_n$ do not lie on the same standard line. 
	We assume no $u_i$ is adjacent to $u_{i+1}$, otherwise we are done. 
	Thus there exists an $i$ such that $u_{i-1}$, $u_i$ and $u_{i+1}$  do not lie in the same standard line. By applying an 
	automorphism of $X$,
	we may assume $u_{i-1}=(1,-1)$, $u_i=(0,0)$ and $u_{i+1}=(1,1)$.  
	As $A\supseteq C$ is saturated,
	Lemma~\ref{lemma_forbidden} ensures that $(1,0)$ is contained in $A$. As $C$ is an $\ell_{\infty}$-component of $A$, $(1,0)\in C$.  
	We are done as $(1,0)$ is adjacent to $u_i$.
\end{proof}

\begin{lemma}\label{lem:comp_is_box}
	If $A\subset X^0$ is a finite saturated set, then every  $\ell_\infty$-component of $A$ is a box.
\end{lemma}
\begin{proof}
	Let $C$ be an $\ell_\infty$-component of $A$.
If $C$ is contained in a standard line, then we are done. Otherwise, by Lemma~\ref{lem:thick box}, $C$ contains two adjacent vertices. In particular, $C$ contains a box that is not contained in a standard line. 
Thus, up to congruence, $C$ contains a box $B=B(a,b,c,d)$ that is not contained in a standard line and is maximal, i.e. is not contained in any other box contained in $C$.
We will show that $B=C$.

Assume for a contradiction that $B$ is a proper subset of $C$. 
As $C$ is
$\ell_\infty$-connected, there exists an $\ell_\infty$-path 
from $B$ to $C \setminus B$ which is contained in $C$. 
Thus, there are vertices $w = (x, y) \in B$ and $v = (x', y') \in C \setminus B$ 
which are $\ell_\infty$-adjacent.
By symmetry, we may assume that $x' + y' > d$. 

We claim that we may assume that $x+y = d$.
For suppose that $x+y \neq d$,  then as $w$ is $\ell_\infty$ adjacent to $v$, we have 
$x+y=d-1$ and $v=(x+1,y+1)$. At least one of $(x,y+1)$ or $(x+1,y)$ is contained in $B$, since $B$ is not contained in a standard line. By replacing $w$ with such a vertex, we may assume $x+y=d$. 

\begin{figure}[H]
	\centering
	\scalebox{.8}{\begin{tikzpicture}
	\begin{pgfonlayer}{nodelayer}
		\node [style=Black] (253) at (0, 0) {};
		\node [style=Black] (254) at (0, -1) {};
		\node [style=Black] (255) at (1, -1) {};
		\node [style=Black] (256) at (-1, -1) {};
		\node [style=Black] (257) at (0, -2) {};
		\node [style=none] (264) at (-2, 2) {};
		\node [style=none] (265) at (3, -3) {};
		\node [style=none] (266) at (-2, 2.5) {$y+x=d$};
		\node [style=Black] (267) at (-1, 1) {};
		\node [style=Black] (268) at (-2, 0) {};
		\node [style=Black] (269) at (-1, 0) {};
		\node [style=none] (271) at (1, -1.6) {$w$};
		\node [style=Black] (272) at (2, -1) {};
		\node [style=none] (273) at (2.6, -1) {$v$};
		\node [style=Black] (274) at (1, 0) {};
		\node [style=none] (275) at (1.6, 0) {$z$};
	\end{pgfonlayer}
	\begin{pgfonlayer}{edgelayer}
		\draw (264.center) to (265.center);
		\draw [style=Dashed] (267) to (268);
		\draw [style=Dashed] (268) to (257);
		\draw [style=Dashed] (257) to (255);
	\end{pgfonlayer}
\end{tikzpicture}}
	
	\caption{}\label{fig:satconnbox}
\end{figure}

We now show that the box $B'\coloneqq B(a,b,c,d+1)$ is contained in $C$, 
contradicting the maximality of $B$.
We first show that $C$ contains a vertex of $B'\setminus B$.  If $v \in B'$, then we are done. 
Otherwise, either $x'+y'=d+2$, or $y'-x'<a$, or $y'-x'>b$. 
If $y'-x'<a$ and $x' + y' \ne d+ 2$ --- as is shown in Figure~\ref{fig:satconnbox} --- then $v=(x+1,y)$ and $y-x=a$. 
As $B$ is not contained in a standard line, $(x-1, y+1) \in A$. Thus by Lemma~\ref{lemma_forbidden}, $z\coloneqq (x, y+1) \in A$. Since $C$ is an $\ell_\infty$-component of $A$, we have that $z \in C\cap (B' \setminus B)$.
If $y' - x' > b$ and $x' + y' \ne d + 2$, then a similar argument demonstrates that $B' \setminus B$ contains a vertex of $C$.
Finally, if $x'+y'=d+2$, then $v=(x+1,y+1)$.
As $B$ is not contained in a standard line, either $(x+1, y-1) \in B$ or $(x-1, y+1) \in B$. 
By Lemma~\ref{lemma_forbidden},
$(x+1, y) \in A$ in the first case and $(x, y+1) \in A$ in the second. 
In either case, using the fact that $C$ is an $\ell_{\infty}$-component of $A$, we deduce there exists a vertex in $C\cap(B' \setminus B)$.

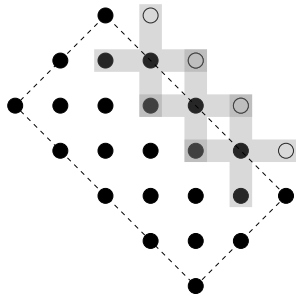
\begin{figure}[htp]
	\centering
	\scalebox{.6}{\begin{tikzpicture}
\foreach \k in {1,...,3} {
\begin{pgfonlayer}{nodelayer}
		\node [style=Empty, fill=white, draw=black]  (1) at (1-\k, 0+\k) {};
		\node [style=Empty, fill=white, draw=black] (2) at (0-\k, 1+\k) {};
	\end{pgfonlayer}
\draw[fill=gray,opacity=0.3,draw=none] (0.25-\k,0.25+\k) -- (1.25-\k,0.25+\k) -- (1.25-\k,-0.25+\k) -- (0.25-\k,-0.25+\k) -- (0.25-\k,-1.25+\k) -- (-0.25-\k,-1.25+\k) -- (-0.25-\k,-0.25+\k) -- (-1.25-\k,-0.25+\k) -- (-1.25-\k,0.25+\k) -- (-0.25-\k,0.25+\k) -- (-0.25-\k,1.25+\k) -- (0.25-\k,1.25+\k) -- (0.25-\k,0.25+\k);
};
\foreach \l in {0,1,2}{
	\foreach \k in {0,1,2,3,4} {
		\begin{pgfonlayer}{nodelayer}
			\node [style=Black] (0) at (0-\k-\l, 0+\k-\l) {};
		\end{pgfonlayer}
	};
};

\foreach \l in {1,2}{
	\foreach \k in {1,2,3,4} {
		\begin{pgfonlayer}{nodelayer}
			\node [style=Black] (0) at (1-\k-\l, 0+\k-\l) {};
		\end{pgfonlayer}
	};
};
\draw[style=dashed] (0,0) -- (-4,4)--(-6,2)--(-2,-2)-- (0,0);

\end{tikzpicture}}
	\caption{Black vertices are in $B$ and white vertices are in $B'\setminus B$, where $B'=B(0,8,0,5)$ and $B=B(0,8,0,4)$. The three overlapping Swiss crosses can be used to deduce, via Lemma~\ref{lemma_forbidden}, that if $C$ contains  a single vertex of $B'\setminus B$, then $B'\subset C$.}\label{fig:swisscross}
\end{figure}

By the previous paragraph, there exists a vertex $z = (x'', y'') \in C \cap (B' \setminus B)$. 
Note that $(x'', y'' -1) \in B$. If $(x'' , y''-2) \in B$, then $(x'' + 1, y''-1) \in C$ by Lemma~\ref{lemma_forbidden}. 
Similarly, if $(x''-1, y'') \in B$, and if $(x''-2, y'') \in B$, then $(x''-1, y''+1) \in C$.
By repeating these arguments, we see that $B' \subset A$ (see Figure~\ref{fig:swisscross}) as claimed,
and we get a contradiction.
\end{proof}
	The following proposition characterizes $\ell_\infty$-connected, saturated subsets of $X^0$ as boxes.
\begin{prop}\label{prop:sat+conn->box}
	A finite subset $A \subset X^0$ is a box if and only if it is $\ell_\infty$-connected and saturated.
\end{prop}
\begin{proof}
	If $A$ is a box, then it is finite, $\ell_\infty$-connected and saturated by Lemma~\ref{lem:box_is_sat}. 
	For the other direction, Lemma \ref{lem:comp_is_box} implies that an $\ell_{\infty}$-connected saturated set is a box.
\end{proof}

The next two lemmas, together with Proposition~\ref{prop:sat+conn->box}, show that a finite minimal saturated subset of $X^0$ is a box (Proposition~\ref{prop:finminsat -> box}). 
\begin{lemma} \label{lem:box_intersection}
Let $C_1$ and $C_2$ be distinct $\ell_\infty$-components of a finite, saturated subset $A\subset X^0$. By Lemma~\ref{lem:comp_is_box}, $C_1$ and $C_2$ are boxes. Let $(x,y)\in\partial C_1\cap \partial C_2$. Up to applying an automorphism of $X$,
we may assume that $(x-1,y)\in C_1$. Then $(x-1,y)$ is a corner of $C_1$ and  $(x+1,y)$ is a corner of $C_2$. Moreover, if $C$ is an $\ell_\infty$-component of $A$ with  $(x,y)\in \partial C$, then $C$ is equal to either $C_1$ or $C_2$.
\end{lemma}
\begin{proof}
Note that $(x,y)\notin A$ and $(x,y\pm1)\notin A$, for otherwise $C_1$ and $C_2$ would be joined by an $\ell_\infty$-path in $A$. Thus it must be the case that $(x+1,y)\in C_2$. In particular, the only vertices of $A$ adjacent to $(x, y)$ are $(x - 1, y)$ and $(x+1, y)$. It follows that the only $\ell_\infty$-components of $A$ which contain  $(x,y)$ in their boundary are precisely $C_1$ and $C_2$.

Moreover,  $(x - 1,y\pm 1) \notin A$ since $A$ cannot contain configuration $(d)$ in Figure~\ref{fig_forbidden} by Lemma~\ref{lemma_forbidden}. We deduce that $(x-1,y)$ and  $(x+1,y)$ are corners of $C_1$ and $C_2$ respectively. 
\end{proof}

For the next lemma, we let $c(A)$ denote the number of $\ell_{\infty}$-components of a subset $A\subset X^0$.

\begin{lemma}\label{lem:saturated_connected_components}
 Let $A$ be a finite saturated subset of $X^0$.
 Then there exists a set $A'\subset X^0$ such that  $\lvert A'\rvert =\lvert A\rvert$ and $\lvert \partial A'\rvert \leq \lvert \partial A\rvert - (c(A)-1)$. If in addition $A$ is minimal, then it is $\ell_\infty$-connected. 
\end{lemma}
\begin{proof}
	We first recall that every $\ell_\infty$-component of $A$ is a box by Lemma \ref{lem:saturated_connected_components}.
We prove the lemma by induction on $c(A)=n$. The base case $n=1$ trivially follows by setting $A'=A$. When $n = 2$, then $A'$ is obtained from $A$ by translating one of the two $\ell_\infty$-components to reduce the boundary of the set by $1$ (this is possible as these components are boxes). Now suppose $c(A) >2$ and that for all finite saturated sets $S$, with $c(S)<c(A)$ there exists a set $S'\subset \cX$ such that  $\lvert S'\rvert =\lvert S\rvert$ and $\lvert \partial S'\rvert \leq \lvert \partial S \rvert - (c(S)-1)$. 

Given a box $B= B(a, b, c, d)$ we say that the line
 $y = -x + d$ is the {\it NE extremal line} of $B$ and that the line
 $y = -x + c$ is a {\it SW extremal line} of $B$. 
Consider the smallest box containing $A$ and let $L$ be its NE extremal line. Let $C$ be an $\ell_{\infty}$-component of $A$
{ which contains a vertex of $L$.}
 Since $A$ is saturated, so is $C$ and hence by Proposition~\ref{prop:sat+conn->box}, the set $C$ is a box.  
Let $\bar C = A \setminus C$. 
By Lemma~\ref{lem:box_intersection}, any vertex of $\partial C \cap \partial \bar{C}$ is adjacent to a corner of $C$ that does not lie on $L$. Thus, $\lvert \partial C \cap \partial \bar C \rvert \leq 2$, and we get that:
$\lvert\partial C\rvert +   \lvert \partial \bar C\rvert \leq \lvert \partial A\rvert +2$.
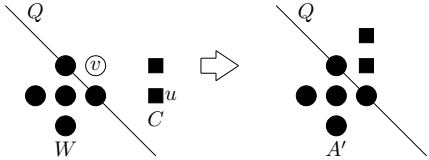
\begin{figure}[htp]
	\centering
	\scalebox{.8}{\begin{tikzpicture}
	\begin{pgfonlayer}{nodelayer}
		\node [style=Empty] (263) at (1, 0) {};
		\node [style=none] (260) at (3, -1.75) {$C$};
		\node [style=none] (252) at (0, -2.75) {$W$};
		\node [style=Black] (253) at (0, 0) {};
		\node [style=Black] (254) at (0, -1) {};
		\node [style=Black] (255) at (1, -1) {};
		\node [style=Black] (256) at (-1, -1) {};
		\node [style=Black] (257) at (0, -2) {};
		\node [style={Black_rectangle}] (258) at (3, 0) {};
		\node [style={Black_rectangle}] (259) at (3, -1) {};
		\node [style=none] (261) at (1, 0) {};
		\node [style=none] (262) at (1, 0) {$v$};
		\node [style=none] (264) at (-2, 2) {};
		\node [style=none] (265) at (3, -3) {};
		\node [style=none] (266) at (-1, 1.75) {$Q$};
		\node [style=none] (267) at (3.5, -1) {$u$};
		\node [style=none] (268) at (4.5, -0.25) {};
		\node [style=none] (269) at (4.5, 0.25) {};
		\node [style=none] (270) at (5, -0.25) {};
		\node [style=none] (271) at (5, -0.5) {};
		\node [style=none] (272) at (5.75, 0) {};
		\node [style=none] (273) at (5, 0.5) {};
		\node [style=none] (274) at (5, 0.25) {};
		\node [style=none] (277) at (9, -2.75) {$A'$};
		\node [style=Black] (278) at (9, 0) {};
		\node [style=Black] (279) at (9, -1) {};
		\node [style=Black] (280) at (10, -1) {};
		\node [style=Black] (281) at (8, -1) {};
		\node [style=Black] (282) at (9, -2) {};
		\node [style={Black_rectangle}] (283) at (10, 1) {};
		\node [style={Black_rectangle}] (284) at (10, 0) {};
		\node [style=none] (287) at (7, 2) {};
		\node [style=none] (288) at (12, -3) {};
		\node [style=none] (289) at (8, 1.75) {$Q$};
	\end{pgfonlayer}
	\begin{pgfonlayer}{edgelayer}
		\draw (264.center) to (265.center);
		\draw (269.center) to (268.center);
		\draw (268.center) to (270.center);
		\draw (270.center) to (271.center);
		\draw (271.center) to (272.center);
		\draw (272.center) to (273.center);
		\draw (273.center) to (274.center);
		\draw (274.center) to (269.center);
		\draw (287.center) to (288.center);
	\end{pgfonlayer}
\end{tikzpicture}}
	
	\caption{Arranging  $W$ and $C$ to get $A'$}\label{fig:reduce_bdry}
\end{figure}
 
We will now show that we can replace $\bar C$ by a different set so that the union of this set with $C$ satisfies the claim.
If $\lvert \bar C \rvert = 1$, then $c(A) = 2$ and we reduce to a base case. 
If $\lvert \bar C \rvert = 2$, then set $W \coloneqq \{(0, 0), (-1, 1)\}$. Otherwise set $W = WW_{\lvert \bar C \rvert}$. Now since $A$ is saturated, $\bar C$ is also saturated and we have $c(\bar C) = c(A)-1 = n-1$. 
By the induction hypothesis, there exists a set $P$ such that $\lvert P \rvert = \lvert \bar{C} \rvert$ and $\lvert \partial P \rvert \le \lvert \partial \bar{C} \rvert - (n-2)$. 
As $W$ is minimal, $|\partial W| \le |\partial P |$. Thus,
 $\lvert \partial W \rvert\leq \lvert \partial \bar C \rvert -(n-2)$.

We now translate $W$ and $C$ appropriately
in order to define $A'$. Let $Q$ be the line given by $y+x=0$. Since $W$ is a Wang--Wang set of size greater than two or is congruent to $\{(0, 0), (-1, 1)\}$, we can apply an automorphism so that $W\subset \{(x,y)\mid y+x\leq 0\}$ and  $\lvert W \cap Q \rvert \geq 2$. Thus there exists some $v\in \partial W$ that is adjacent to two vertices of $W \cap Q$.
 Let $u \in C$
be a vertex on the SW extremal line of $C$. 
By translating $C$, we can suppose that $u$ and $v$ coincide. Setting $A' \coloneqq W \cup C$, we get (see Figure~\ref{fig:reduce_bdry})
\[\lvert \partial A' \rvert \le \lvert \partial C \rvert + \lvert \partial W \rvert - 3.\] Thus
\[ \lvert \partial A' \rvert \le \lvert \partial C \rvert + \lvert \partial W \rvert - 3 \le \lvert \partial C \rvert + \lvert \partial \bar C \rvert  - (n - 2) - 3  \le \lvert \partial A \rvert - (n - 1) \]
and
\[\lvert A'\rvert  = \vert W \rvert + \lvert C \rvert  = \lvert \bar C \vert + \lvert C \rvert = \lvert A \rvert. \] 
This concludes the proof of the lemma. 
\end{proof} 

\begin{proof}[Proof of Proposition \ref{prop:finminsat -> box}]
 Suppose $A\subset X^0$ is minimal and saturated. By Lemma~\ref{lem:saturated_connected_components}, $A$ is $\ell_\infty$-connected.  Lemma \ref{lem:comp_is_box} now implies that $A$ is a box.
\end{proof}

    \section{Minimal sets are connected}\label{sec:connected}
\emph{  Recall that a  subset $A\subset X^{0}$ is \emph{connected} if the subgraph of $X$ induced by $A$ is connected.
In this section we prove that (almost) all minimal sets are connected. Additionally, we use similar ideas to prove another characterization of minimal sets in terms of cones.}
\vspace{.3cm}

  \begin{thm} \label{thm:connected} 
\connected
  \end{thm}

Before proving the above theorem, we need to first define cones and prove a series of lemmas.

\begin{definition} \label{def_cones}
	A {\it cone} is a subset of $X^0$ congruent to $C_0 := \{(x,y) \, | \, y-x \geq 0, y+x \geq 0\}$. An {\it extremal ray of the cone $C_0$} is the intersection of $C_0$ with either the line $y=x$ or $y=-x$. 
	An {\it extremal ray of a cone} is the image of an extremal ray of $C_0$ under the given congruence.
	A {\it cone at the vertex $v \in X^0$} is a cone whose two extremal rays intersect at $v$. 
	If $v \in X^0$, then the cone {\it above} $v$ is a cone at $v$ that is translation-equivalent to~$C_0$; the cone {\it to the right of $v$} is a cone at $v$ that is translation-equivalent to $C_0$ rotated clockwise $90^\circ$. The cones {\it below $v$} and {\it to the left of $v$} are defined analogously. 
\end{definition}

\begin{lemma}\label{lem:cone char}
Let $A\subset X^0$ be a finite set such that  $\lvert \enc(A)\setminus A\rvert < \Exc(\enc(A))$ and $X^0\setminus A$ is a union of cones. Suppose $v\in A$ and there is a cone $C$ based at $v$ such that $C\cap A=\{v\}$. Without loss of generality, we can suppose $v=(0,0)$ and that $C$ is the  cone  above $v$. Then $(-1,-1),(0,-1),(1,-1)\in A$.
\end{lemma}
\begin{proof}
	Suppose $B=B(a,b,c,d)\coloneqq \enc(A)$. Let $A'\coloneqq A\setminus \{v\}$. Observe that $X^0\setminus A'$ is a union of cones, $C$ is disjoint from $A'$ and that $1\leq\lvert B\setminus A'\rvert \leq \Exc(B)$.
	 We will show that $w_{-1}=(-1,-1)$, $w_0=(0,-1)$ and $w_1=(1,-1)$ are in $A$. By symmetry, we need only show $w_0$ and $w_1$ are in $A$.

	\underline {$w_0\in A$.} If $w_0\notin A$, then $w_0$ is contained in a cone $C_0$ disjoint from $A$, which we may assume is at $w_0$. Since $v\notin C_0$, $C_0$ must either be below, to the left or to the right of $w_0$. In either of the latter two cases, $C_0\cup C$ contains either the intersection of the extremal line  $y-x=b$ with $B$ or the intersection of the extremal line $y+x=d$ with $B$. 
	As, $\lvert (C_0\cup C)\cap B\rvert \leq \lvert B\setminus A'\rvert\leq \Exc(B)$,
	 this contradicts Lemma \ref{lemma_exc_bound}(2). Thus we may assume that $C_0$ is the cone below~$w_0$.
	
	Note that $v$ does not lie on the extremal line $y-x = a$ nor the extremal line $y+x = c$, for otherwise $C$ contains the intersection of an extremal line with $B$ and we get a contradiction as in the previous paragraph. 
	Let $L$ be the line of slope $-1$ passing through  $v$. 
	Let $\phi:X^0\rightarrow X^0$ be the translation $(x,y)\mapsto (x-1,y)$.
	As $v$ does not lie on $y-x = a$ or $y+x = c$, given a vertex $u\in L\cap B$ it follows that either  $u\in C\cap B\cap L$ or $\phi(u)\in C_0\cap B$. 
	 It follows that \begin{align}\label{eqn:cone bound}
	\lvert B\cap L \rvert \leq \lvert C\cap B \cap L\rvert+\lvert C_0\cap B\cap L\rvert\leq \lvert C\cap B\rvert+\lvert C_0\cap B\rvert\leq \lvert B\setminus A'\rvert \leq \Exc(B).
	\end{align} 
	As $v$ does not lie on $y-x = a$ or $y+x = c$, either $L$ is an extremal line of $B$ or $(0,1)\in B$. 
	If $L$ is an extremal line, then this contradicts Lemma \ref{lemma_exc_bound}(2). If not, then $(0,1)\in C\cap (B\setminus L)$, so $\lvert C\cap L \cap B\rvert<\lvert C\cap  B\rvert$. Thus the inequality in  (\ref{eqn:cone bound}) is strict, contradicting Lemma \ref{lemma_exc_bound}(3). We deduce that $w_0\in A$.
	
	\underline{$w_{1}\in A$.} Suppose $w_1\notin A$. Then there is a cone $C_1$  based at $w_1$ that does not intersect $A$. Since $v\notin C_1$, $C_1$ must be either below or  to the right of  $w_1$. In either case, if $L$ is the line through $v$ and $w_1$, then $B\cap L\subseteq C\sqcup C_1$ and so \begin{equation}\label{eqn:w1 cone calc}
	\lvert B\cap L\rvert \leq \lvert C_1\cap B \cap L\rvert+\lvert C\cap B \cap L\rvert\leq \lvert C_1\cap B\rvert+\lvert C\cap B\rvert\leq \lvert B\setminus A'\rvert \leq \Exc(B).
	\end{equation} 
	If $L$ an extremal line of $B$, then (\ref{eqn:w1 cone calc}) contradicts Lemma \ref{lemma_exc_bound}(2). We thus assume $L$ is not an  extremal line of $B$.  
	
	We claim that the inequality in (\ref{eqn:w1 cone calc}) is strict. If we show this, then 
	 we get a contradiction by Lemma \ref{lemma_exc_bound}(3),  and we can deduce that $w_1\in A$ as required. 
	We first observe that as $L$ is not an extremal line of $B$ and as $w_0 = (0,-1) \in B$ (by what we have already shown), we must have that $(1,0)\in B$. 
	There are now two cases depending on whether or not  $w_1\in B$. If $w_1\notin B$, then the line $y-x=-1=a$ through  $(0,-1)$ and $(1,0)$ is an extremal line of $B$.
	Furthermore, as $\Exc(B) > 0$, it follows from Lemma \ref{lemma_exc_bound}(2) that $b-a \ge 2$. 
	Thus, we conclude that $(0,1)\in (B\cap C)\setminus L$. 
	We thus deduce that  $\lvert C\cap L \cap B\rvert< \lvert C\cap B\rvert$ and that the inequality in (\ref{eqn:w1 cone calc}) is strict, proving the claim when $w_1\notin B$.

	If $w_1\in B\setminus A$, then $\Exc(B)\geq \lvert B\setminus  A\rvert+1\geq 2$.
	Lemma \ref{lemma_exc_bound}(2) then implies that every extremal line of $B$ contains at least three vertices, and so $b-a,d-c\geq 4$. As $v,w_1,(0,-1),(1,0)\in B$, it follows that $B\setminus L$ must intersect at least one of $C$ or $C_1$, and so either $\lvert C\cap L \cap B\rvert< \lvert C\cap B\rvert$ or $\lvert C_1\cap L \cap B\rvert< \lvert C_1\cap B\rvert$. In either case, we deduce as before that the inequality in (\ref{eqn:w1 cone calc}) is strict as required.
\end{proof}

\begin{lemma}\label{lem:disconnectedbox}
	Up to congruence,  $B(0,2)$ is the only disconnected box that is a minimal set. In particular, if $B$ is a box with $\Exc (B)>0$, then $B$ is connected.
\end{lemma}
\begin{proof}
	Suppose $B$ is  a disconnected box that is a minimal set. Then it  contains more than one vertex and is contained in a standard line. 
	Thus $B$ is congruent to $B(0,2n)$ for some $n>0$. Note that $\Exc(B)=\frac{n(1-n)}{2}$ by Theorem \ref{thm:box_excess}.   Lemma \ref{lem:+ve exc minimal} implies  $n=1$ and so $B = B(0,2)$. Furthermore, as $\Exc(B)=0$, the second claim follows. 
\end{proof}
The next lemma describes the geometry of a minimal set.

\begin{lemma} \label{lemma_removing_cone}
	Suppose $A\subset X^0$ is minimal. Then $X^0\setminus A$ is a union of cones. Moreover,  $A$ is connected if and only if $\enc(A)$ is connected.
\end{lemma}
\begin{proof}
	We proceed by induction on $n\coloneqq \lvert \enc(A)\setminus A\rvert$. The base case $n=0$ is clear, since  then $A=\enc(A)$ is a box and the complement of a box is a disjoint union of cones. 
		For our inductive hypothesis, we assume the conclusion  holds for all minimal sets $A$ with $n=\lvert \enc(A)\setminus A\rvert$. 
		
	Suppose $A'\subset X^0$ is a minimal set with $\lvert \enc(A')\setminus A'\rvert=n+1\geq 1$. Let $B\coloneqq \enc(A')$. By Theorem \ref{thm:minimal set characterization}, as $A'$ is minimal and $\lvert B\setminus A'\rvert \geq 1$, we have $\Exc(B)>0$. Thus Lemma \ref{lem:disconnectedbox} implies $B$ is connected.
	 By Lemma \ref{lem:nested_sequence_to_enc_box}, there exists a minimal set $A$ such that $\lvert A\rvert = \lvert A'\rvert+1$ and $A'\subsetneq A\subseteq B$. 
	 In particular, $\enc(A) = B$ and, by Theorem \ref{thm:minimal set characterization}, $|\partial A| = |\partial A'| = |\partial B|$.
	 As $\lvert B\setminus A\rvert =n$ and $B$ is connected,  the inductive hypothesis tells us that  $A$ is connected and $X^0\setminus A$ is a union of cones. Let $v\in A\setminus A'$. 	 We will show that $v$ is contained in a cone $C$ disjoint from $A'$.

	 Since  $\lvert \partial A\rvert=\lvert \partial A'\rvert$, $A' \subset A$, $|A| = |A'| + 1$ and $A$ is connected,
	 there exists some vertex $w \in \partial A$ adjacent to $v$ which is not in $N(A')$.
	 In particular, $w$ is adjacent to $v$ and no other  vertex in $A$.
	 Without loss of generality, we can suppose $v=(0,0)$ and $w=(0,1)$.  Suppose $R^+$ and $R^-$ are the rays $\{(x,y)\mid y-x= 0,y\geq 1\}$	 and $\{(x,y)\mid y+x= 0,y\geq 1\}$ respectively. Since $b\coloneqq (1,1)$  is adjacent to $w$, it is  not contained in $A$. As $X^0\setminus A$ is a union of cones, $b$ is contained in a cone $C'$ that is disjoint from $A$. Since $C'$ cannot contain $v$, it must contain the ray $R^+$, so that $R^+\cap A=\emptyset$. A similar argument using $(-1,1)$ allows one to deduce $R^-\cap A=\emptyset$.	 
	 Let $C$ be the cone above $v$. 
	 As $A$ is connected and  $\{w\} \cup R^-\cup R^+$ do not contain vertices of $A$, it follows that $C \cap A' = \emptyset$.
	Since $X^0\setminus A$ is a union of cones and $v\in C$, it follows that $X^0\setminus A'$ is also a union of cones. This proves the first claim of the lemma.
	 
	  Since $A'$ is minimal, we have $\lvert B\setminus A'\rvert \leq \Exc(B)$, and so $\lvert B\setminus A\rvert \leq \Exc(B)-1$. Thus Lemma \ref{lem:cone char} can be applied to $A$, $v$ and $C$ to deduce that $(-1,-1),(0,-1),(1,-1)\in A'$. We now show $A'$ is connected. We pick arbitrary $p,q\in A'$ and show that $p$ and $q$ can be joined by a path in $A'$. Since $A$ is connected, there exists a simple path $P=(p=v_0, v_1, \dots, q= v_m)$ in $A$. If $P$ is a path in $A'$ we are done. If not, then $v_i=(0,0)$ for some $0<i<m$. Since $P$ is simple, $v_j\in A'$ for all $j\neq i$.  Thus $v_{i-1}$ and $v_{i+1}$ are contained in $A'$ and are adjacent to $v$, and so $v_{i-1},v_{i+1}\in\{(-1,0),(1,0),(0,-1)\}$. Since $(-1,-1),(0,-1),(1,-1)\in A'$, $v_{i-1}$ and  $v_{i+1}$ can be joined by a path $(v_{i-1}=u_0,u_1, \dots, u_t=v_{i+1})$ in $A'$. Thus  $(p=v_0,\dots, v_{i-1}=u_0,\dots, u_t=v_{i+1},\dots v_m=q)$ is a path from $p$ to $q$ in $A'$.
\end{proof}

\begin{proof}[Proof of Theorem \ref{thm:connected}]
	Suppose $A$ is a minimal set that is not connected. Lemma \ref{lemma_removing_cone} implies $\enc(A)$ is not connected. By Lemma \ref{lem:disconnectedbox},  $\enc(A)$ is congruent to $B(0,2)$. Since  $\Exc(B(0,2))=0$, Theorem \ref{thm:minimal set characterization} ensures  $A=\enc(A)$  and so $A$ is  congruent to $B(0,2)$.
\end{proof}

We now state a characterization of minimal sets using cones. 

\begin{thm}\label{thm:minimal set characterization_refined}
	Let $A\subset X^0$ with $N := \left| \enc(A)\setminus A\right|$ and $E := \Exc(\enc(A))$. Then the following are equivalent:
	\begin{enumerate}
		\item\label{item:coneinert0} $A$ is minimal;
		\item\label{item:coneinert1} $\left| \partial A\right| =\left| \partial (\enc(A))\right |$ and $N \le E$;
		\item\label{item:coneinert2} $X^0\setminus A$ is a union of cones and $N \le E$.
	\end{enumerate}
\end{thm}
\begin{proof}
The equivalence of (\ref{item:coneinert0}) and (\ref{item:coneinert1}) follows from Theorem \ref{thm:minimal set characterization}. Theorem \ref{thm:minimal set characterization} and Lemma \ref{lemma_removing_cone} tells us that (\ref{item:coneinert0}) implies (\ref{item:coneinert2}). All that remains is to show that (\ref{item:coneinert2}) implies  (\ref{item:coneinert0}).

Let $A\subset X^0$ with $\left| B\setminus A\right| \leq \Exc(B)$, where $B\coloneqq \enc(A)$. Suppose $X^0\setminus A$ is a union of cones. Since $\Exc(B)\geq 0$, $B$ is minimal by Lemma \ref{lem:+ve exc minimal}. Let $D$ be a minimal set such that $A\subset D\subseteq B$, with $|D|$ 
minimal among all such choices. Notice that $\enc(D)=B$. 

For contradiction, suppose $A$ is not minimal. Then $A\subsetneq D$, so pick $v\in D\setminus A$. As $X^0\setminus A$ is a union of cones, there exists a cone $C'$ containing $v$ and disjoint from $A$. Without loss of generality, we may suppose $C'$ faces upwards (i.e. it lies above some vertex) and that $v\in D\cap C'$ has maximal $y$-coordinate out of all vertices in $D\cap C'$. Let $C$ be the cone above $v$. Since $C\subset  C'$ and $v$ has maximal $y$ coordinate,  $C$ is a cone above $v$ such that $C\cap D=\{{v}\}$. 

Since $D$ is a minimal set,  Lemma \ref{lemma_removing_cone} ensures $X^0\setminus D$ is a union of cones. Moreover, since $A\subsetneq D\subseteq B$,  $\lvert B\setminus D\rvert<\lvert B\setminus A\rvert\leq \Exc (B)$. Without loss of generality, we can assume that $v=(0,0)$. Lemma \ref{lem:cone char} now implies that $(-1,-1)$, $(0,-1)$ and $(1,-1)$ are in $D\setminus \{v\}$.
As $(-1,-1), (0,-1), (1,-1) \in D$ and $C \cap D = \{v\}$, it follows that  $\lvert \partial(D\setminus \{v\})\rvert = \lvert \partial D\rvert$.
 Theorem~\ref{thm:minimal set characterization}  
 thus ensures $D\setminus \{v\}$ is minimal. Since $A\subseteq D\setminus \{v\}\subsetneq D\subseteq B$, this contradicts our choice of $D$. Thus $A$ is minimal as required.
\end{proof}

\section{The graph of minimal sets} \label{sec:graphofminsets}

\emph{In this section, we study the graph of minimal sets, $\mathcal{G}$ (see the introduction for a definition).}
\vspace{.3cm}


\subsection{Dead and mortal sets}
A finite subset $A\subset X^0$ is \emph{efficient} if for every $B \subset X^{0}$, $\lvert \partial B\rvert=\lvert \partial A\rvert$ implies $\lvert A\rvert \geq \lvert B\rvert$, and we say that $A$ is \emph{inefficient} otherwise.
Equivalently, $A$ is efficient if  $\lvert \partial B\rvert=\lvert \partial A\rvert$ implies that $\Exc(A) \ge \Exc(B)$  by Lemma~\ref{lemma_def_same_boundary}.
The main result of this subsection is Theorem~\ref{thm:dead_char2}, which characterizes dead sets (defined in the introduction) as inefficient sets, which in turn are classified in terms of specific boxes (Lemma~\ref{lem:boxefficient}),
 and Theorem~\ref{thm:mortalchar} characterizing mortal sets.

We first show that efficient sets are boxes and are minimal:
 \begin{lemma}\label{lem:dead_sat+box}
Every efficient set is minimal, saturated and a box.   
 \end{lemma}
\begin{proof}
Let $A$ be an efficient set. 
By Lemma~\ref{lem:excess_well_defined}, 
there exists  a minimal set $A'$
 such that $|\partial A'| = |\partial A|$. Since $A$ is efficient, we have $|A| \geq |A'|$. 
If $A$ is not minimal, then there exists a minimal set $A''$ such that $\lvert A'' \rvert = \lvert A \rvert \geq|A'| $ and 
$|\partial A''| < |\partial A| = |\partial A'|$. 
This contradicts Lemma~\ref{lem:bdry monotone}, so $A$ must be minimal. If $A$ is not saturated, then there exists a vertex $v$ such that $|\partial (A\cup \{v\})| \leq |\partial A|$,  which  contradicts Lemma~\ref{lem:bdry monotone} and the fact that $A$ is minimal and efficient.
Proposition~\ref{prop:finminsat -> box} now implies that every efficient set is a box.
\end{proof}

We say $A\subset X^0$ is a \emph{Wang--Wang box} if it is simultaneously a box and a  Wang--Wang set.  

\begin{remark} \label{rmk:wang_wang_boxes}
	A subset of $X^0$ is a  Wang--Wang box if and only if it is congruent to either  $B(m,m)$ or $B(m,m+1)$ for some  $m\in \bZ$ (recall Notation~\ref{notation:boxes}). 
\end{remark}

We now characterize efficient sets.

\begin{lemma}\label{lem:boxefficient}
A subset of $X^0$ is efficient if and only if it is congruent to either a  Wang--Wang box or  $B(m-1,m+1)$ for some odd $m\in \bN$.
\end{lemma}
\begin{proof}
  Let $B$ be an efficient set. By Lemma~\ref{lem:dead_sat+box}, $B$ is a box. 
   Remark~\ref{rmk:B vs B hat} implies that no box of the form  $\hat{B}(\alpha, \beta)$ is  efficient. 
It thus follows from 
	Proposition~\ref{prop:charbox}
 that $B$ is congruent to  $B(\alpha, \beta)$  for some $\alpha, \beta \in \mathbb{Z}$.
Without loss of generality, we may assume that $\beta \ge \alpha$, and we 
set $r = \frac{\alpha + \beta}{2}$ and $k = \frac{\beta - \alpha}{2}$.
Note that $\alpha = r - k$ and $\beta = r + k$. 

Suppose first that  $\alpha$ and $\beta$ have the same parity. Then $r$ and $k$ are both integers.
By Lemma~\ref{lem:bdry}, $\lvert\partial B(r-k,r+k)\rvert=\lvert\partial B(r,r)\rvert$.
Since $B$ is efficient, Theorem~\ref{thm:box_excess} gives:
\begin{equation*}
\Exc(B(r-k,r+k)) = \left \lfloor \frac{ \lfloor r \rfloor -k^2}{2} \right \rfloor\geq \Exc (B(r,r))= \left \lfloor \frac{ \lfloor r \rfloor }{2} \right \rfloor. \end{equation*}

By the above equation, either $k=0$, or $r$ is odd and $k=1$. Thus $B$ is either the Wang--Wang box $B(\alpha,\alpha)$ or the box $B(m-1, m+1)$ for $m$ odd, respectively. 

On the other hand, suppose  $\alpha$ and $\beta$ have different parity. Then $r = s+\frac{1}{2}$ and $k = t+\frac{1}{2}$ for some  $s,t\in \bN$. By Lemma~\ref{lem:bdry}, $\lvert\partial B(r-k,r+k)\rvert=\lvert\partial B(r-1/2,r+1/2)\rvert$.
By applying Theorem~\ref{thm:box_excess} and using that $B$ is efficient and $t^2 + t$ is even, we get the following:
\begin{equation*}
	\Exc(B(r-k,r+k))
	=\left \lfloor \frac{ s -t^2-t-\frac{1}{4}}{2} \right \rfloor
	=\left \lfloor \frac{ s -\frac{1}{4}}{2} \right 
	\rfloor-\frac{t^2+t}{2}	
	\geq \Exc (B(r-1/2,r+1/2))
	\geq \left \lfloor \frac{ s -\frac{1}{4}}{2} \right \rfloor 
\end{equation*} 
Thus $t^2+t\leq 0$, which implies $t=0$ and hence $\beta = \alpha + 1$. Thus $B$ is congruent to $B(\alpha,\alpha+1)$, which is a Wang--Wang box.

For the converse, 
suppose we are given a Wang--Wang box $WW_n$.
By Lemma~\ref{lem:box_is_sat}, a box is saturated. This implies $\lvert\partial WW_{n+1}\rvert > \lvert \partial WW_n\rvert$. 
Since $WW_{n+1}$ is minimal,  Lemma~\ref{lem:bdry monotone} ensures that  given any $B \subset X^0$ with $\lvert B\rvert> \lvert WW_{n}\rvert$ then   $\lvert\partial B\rvert>\lvert\partial WW_{n}\rvert$. Thus a Wang--Wang box $WW_n$ is efficient. 
By Lemma~\ref{lem:bdry} and Theorem~\ref{thm:box_excess}, for odd $m$, we have $\lvert\partial B(m-1,m+1)\rvert=\lvert\partial B(m,m)\rvert$ and $\Exc(B(m-1,m+1))=\Exc(B(m,m))$.
As $B(m,m)$ is efficient, $B(m-1, m+1)$ is also.  
\end{proof}

Before proving the next result, we first show that one can always add a vertex to a box such that the resulting
set has boundary one larger than the box, as long as the box contains at least
two vertices.

\begin{lemma} \label{lem:box_increase_bdry}
	Let $B \subset X^0$ be a box containing at least two vertices. Then $\lvert \partial ( B \cup \{v\}) \rvert
	= \lvert \partial B \rvert + 1$
	for some $v \in X^0 \setminus B$.
\end{lemma}
\begin{proof}
	Let $(x, y)$ be a vertex of $B$ with $y$ maximal. 
	Suppose first that $(x+1, y) \in B$. Since $B$ is saturated by Lemma~\ref{lem:box_is_sat}, it follows from Lemma~\ref{lemma_forbidden} that  $(x-1, y) \notin B$.
	The claim now follows for this case by noting that $\lvert \partial \big( B \cup \{(x, y+1)\} \big) \rvert
	= \lvert \partial B \rvert+1$.
	A similar argument shows the claim when $(x-1, y) \in B$. 
	
	Next consider the case where $(x-1, y) \notin B$ and $(x+1, y) \notin B$.
	Suppose first that $(x, y-1) \in B$. 
	As $B$ is a box with $(x,y) \in B$ and $(x+1, y) \notin B$, it follows that
			$(x+2, y) \notin B \cup \partial B$. 
	The claim now follows by noting that $\lvert \partial \big(B \cup \{(x+1, y)\}
	\big) \rvert = \lvert \partial B \rvert + 1$.
	Finally, suppose that $(x, y-1) \notin B$. 
	In this case, as $B$ is a box and $|B| \ge 2$, we must have that either
	$(x-1, y-1) \in B$ or $(x+1, y-1) \in B$. Without loss of generality, suppose
	the former is true. As $B$ is saturated, Lemma~\ref{lemma_forbidden} implies that  $(x-2, y-1) \notin
	B$. It now follows that $\lvert \partial \big( B \cup \{(x-1, y)\} \big) \rvert = \lvert
	\partial B \rvert + 1$.
\end{proof}

We now give a characterization of dead sets. 

\begin{thm}\label{thm:dead_char2} Let $A\subset X^0$ be a minimal set. The following are equivalent: 
\begin{enumerate}
 \item $A$ is dead. 
 \item $A$ is an inefficient box. 
 \item $A$ is a box that is not congruent to either a Wang--Wang box or to a box of the form $B(m-1,m+1)$ for odd $m\in \bZ$.
\end{enumerate} 
\end{thm}

\begin{proof}
	Suppose $A$ is dead. We first show that $A$ is saturated. Let $v\in X^0\setminus A$. If $\left| \partial (A\sqcup \{v\}) \right|\leq \left| \partial A \right|$, then Lemma~\ref{lem:bdry monotone} would imply that $A\sqcup  \{v\}$ is minimal, contradicting the hypothesis that $A$ is dead. Thus $\left| \partial (A\sqcup \{v\}) \right|> \left| \partial A \right|$ for all $v\in X^0\setminus A$, ensuring that $A$ is saturated. 
	Since  $A$ is saturated and minimal,  Proposition~\ref{prop:finminsat -> box} implies it is a box. 
	Since a set with one vertex is not dead, $|A| >1$. Hence Lemma~\ref{lem:box_increase_bdry} ensures that there exists a vertex $v \in X^0 \setminus A$ such that $\lvert\partial ( A \cup \{v\})\rvert = \lvert\partial A\rvert + 1$.
	Since $A$ is dead, $A \cup \{v\}$ is not minimal and so  there exists a minimal set $C$ such that $|C| = |A|+1$ and $|\partial C| < |\partial A|+1$. As $A$ is minimal, Lemma~\ref{lem:bdry monotone} implies $|\partial C| = |\partial A|$. Thus $A$ is an inefficient box. 
	
	Now suppose $A$ is an inefficient box. Since $A$ is saturated by Lemma~\ref{lem:box_is_sat}, $\lvert \partial(A\cup \{v\})\rvert >\lvert \partial A\rvert$ for every $v\in X^0\setminus A$. As $A$ is inefficient, there exists a set $C$ such that $\lvert C\rvert >\lvert A\rvert$ and $\lvert \partial C\rvert=\lvert \partial A\rvert$. Lemma~\ref{lem:bdry monotone} now implies 
	that $A\cup \{v\}$ isn't minimal for any $v\in X^0\setminus A$.
	Hence, $A$ is dead. 
	
    The equivalence of (2) and (3) is shown in Lemma~\ref{lem:boxefficient}.
\end{proof}


Finally, we characterize mortal sets:
\begin{thm}\label{thm:mortalchar}
	A minimal set is mortal if and only if its enclosing box is dead.
\end{thm}
\begin{proof}
	Let $A$ be a minimal set. 
	We first show that $A$ is mortal if and only if $\enc(A)$ is mortal.
	By Lemma~\ref{lem:nested_sequence_to_enc_box}, 
	there exists a sequence $A = A_0 \subset A_1 \subset \dots \subset \enc(A)$ of nested minimal sets such that $|A_{i+1}| = |A_{i}| + 1$. Thus, if $A$ is mortal, so is $\enc(A)$. On the other hand, suppose $\enc(A)$ is mortal. By Lemma~\ref{lem:nested_sequence_to_enc_box}, any maximal nested sequence $A = A_0' \subset A_1' \subset \dots $ of minimal sets with $|A_{i+1}'| = |A_i'| + 1$ must include $\enc(A)$ and, in particular, must be finite as $\enc(A)$ is mortal. Thus, $A$ is mortal.
	
	Consequently, in order to prove the theorem, we need to show that a box $B$ is mortal if and only if it is dead. Since dead sets are mortal, this reduces to demonstrating that a box which is not dead is immortal. By Theorem~\ref{thm:dead_char2}  we only need to show that Wang--Wang boxes and $B(m-1, m+1)$, for odd $m \in \mathbb{Z}$, are immortal sets. 
	Wang--Wang boxes are immortal because they are contained in the infinite nested sequence of minimal sets $(WW_n)_{n=1}^\infty$. 

	Let $B=B(m-1,m+1)$ for some odd $m\in \bZ$. Let $v=(0,m)$ and $B'\coloneqq B\sqcup \{v\}$. Since $v$ is not contained in $B$ but is adjacent to $(0,m-1),(1,m)\in B$, it follows that $\lvert \partial B'\rvert=\lvert \partial B\rvert+1$.  By Lemma~\ref{lem:bdry}, Lemma~\ref{lem:area} and as $m$ is odd, $\lvert B\rvert =\lvert B(m,m)\rvert$ and  $\lvert \partial B\rvert=\lvert \partial B(m,m)\rvert$. Since $B(m,m)$ is  a Wang--Wang box and in particular,  it is saturated, then any minimal set of size $\lvert B\rvert+1$ must have boundary strictly greater than $\lvert \partial B\rvert$. Thus, $B'$ is a minimal set.
	By Lemma~\ref{lem:nested_sequence_to_enc_box}, there exists a sequence of minimal sets $B \subset B' \subset \dots \subset \enc(B')$ such that the size of the symmetric difference between consecutive sets in this sequence is one. As $\enc(B') = B(m, m+1)$ is a Wang--Wang set (see Remark~\ref{rmk:wang_wang_boxes}), it is immortal. Thus, $B$ is immortal as well.
\end{proof}

\subsection{Uniquely minimal sets}
  In this subsection, we characterize uniquely minimal sets in $X$.
  Recall from the introduction that the \textit{grading} of a vertex of $\mathcal{G}$ is the size of one of its representatives, and uniquely minimal sets correspond exactly to vertices of $\mathcal{G}$ that are unique
  out of vertices of the same grading.


\begin{lemma}\label{lem:wwset not uniq min}
	Let $WW_n$ be a Wang--Wang set that is not a box. Then there exists a minimal set $A$ such that $\lvert A\rvert =n$ and $A$ is not congruent to $WW_n$.
\end{lemma}
\begin{proof}
	The box $B\coloneqq \enc(WW_n)$ is congruent to $B(\alpha,\beta)$ where either $\alpha= \beta$ or $\alpha+1=\beta$. By Lemma~\ref{lem:special box1 excess} 
	and Theorem~\ref{thm:minimal set characterization}, $k\coloneqq \lvert B\rvert - \lvert WW_n\rvert\leq \Exc(B)\leq \frac{\alpha}{2}$. 
	Since $k\geq 1$ (as $WW_n$ is not a box),
	Lemma~\ref{lem:special box1 excess} implies that either $\alpha\geq 3$ or $\alpha=\beta=2$ (indeed, $\Exc(B(2,3)) =0$). 
	In the latter case, $k=1$ and $n=4$, 
	so we observe that $\lvert \hat B(2,2)\rvert =\lvert WW_4\rvert$
	 and $\lvert \partial \hat B(2,2)\rvert =\lvert \partial WW_4\rvert$. 
	Since $\hat B(2,2)$ is not congruent to $WW_4$, we are done.
	
	Therefore, we may assume $\alpha\geq 3$. 
	Let $B'\coloneqq B(\alpha-1,\beta+1)$. 
	By Lemma~\ref{lem:bdry}, $\lvert \partial B\rvert=\lvert \partial B'\rvert$.
	By Lemma~\ref{lem:area},
	\[\lvert B\rvert =\left\lfloor\frac{\alpha \beta +\alpha+\beta+2}{2}\right\rfloor\]
	\[\lvert B'\rvert =\left\lfloor\frac{(\alpha-1)(\beta+1)+\alpha+\beta+2}{2}\right\rfloor
	=\left\lfloor\frac{\alpha\beta+2\alpha+1}{2}\right\rfloor.\]
	 Thus, we have that  $\lvert B\rvert-1\leq \lvert B'\rvert\leq\lvert B\rvert$ in both the case that $\alpha = \beta$ and that $\beta = \alpha + 1$.
	 
	 Note that the line $y = x$ contains  $\lfloor \frac{\beta + 1}{2}+1\rfloor$ vertices of $B'$.
	 Since $k\leq\frac{\alpha}{2}< \frac{\beta+1}{2}$, by Lemma~\ref{lemma_exc_bound}  there exists some set $A$ such that $\lvert \partial A\rvert =\lvert \partial B\rvert=\lvert \partial WW_n\rvert$, $\lvert A\rvert =\lvert B\rvert-k=\lvert WW_n\rvert$, and $\enc(A)=B'$. In particular, $A$ must be minimal. Since $B'=\enc(A)$ and $B=\enc(WW_n)$ are not congruent, $A$ and $WW_n$ are not congruent. 
\end{proof}

\begin{thm}\label{thm:unique_min}
	\uniquemin
\end{thm}

\begin{proof}
	Suppose $A$ is uniquely minimal. Since $WW_{\lvert A\rvert}$ is minimal and   $\lvert WW_{\lvert A\rvert}\rvert=\lvert A\rvert$, $A$ must be congruent to $WW_{\lvert A\rvert}$. By Lemma~\ref{lem:wwset not uniq min}, $A$ must be a box. We note that $A$ cannot be congruent to $B(r,r)$ for odd $r\in \bN$, since  Lemma~\ref{lem:bdry} and Lemma~\ref{lem:area} imply that $\lvert B(r-1,r+1)\rvert= \lvert B(r,r)\rvert$ and $\lvert \partial B(r-1,r+1)\rvert=\lvert \partial B(r,r)\rvert$. Thus $A$ is a  box of the form $B(2n,2n)$ or $B(n,n+1)$ for some $n\in \bN$ by Remark~\ref{rmk:wang_wang_boxes}.
	
	For the converse, suppose $B$ is congruent to $B(2n,2n)$ or $B(n,n+1)$ for some $n\in \bN$. In particular, $B$ is congruent to a Wang--Wang set, so it is minimal.  Suppose $A$ is another minimal set such that $\lvert A\rvert=\lvert B\rvert$. It follows that $\lvert \partial A \rvert = \lvert \partial B \rvert$. By Lemma~\ref{lem:boxefficient} $B$ is efficient. Since $\lvert A\rvert=\lvert B\rvert$, we get that $A$ is also efficient. Furthermore, Lemma~\ref{lem:boxefficient} also implies that any efficient set of size $\lvert B\rvert$ is actually congruent to $B$, hence $A$ is congruent to $B$.
\end{proof}

\begin{cor} \label{cor_one_inf_comp}
	The graph $\mathcal{G}$ contains exactly one infinite connected component.
\end{cor}
\begin{proof}
	Let $\cC$ be an infinite component of $\cG$.
	As there are only finitely many sets (up to congruence) of any given size, there exists a number $m_0$ such that $\cC$ contains a vertex of grading $m$ for every $m\geq m_0$.  By Theorem~\ref{thm:unique_min}, $B(2n,2n)$ is uniquely minimal for every $n$. Thus, $\cC$ contains 
	$B(2n,2n)$ for every $n$ sufficiently large, and so  $\cC$ is the unique infinite component of $\cG$.
\end{proof}

\subsection{Finite components} \label{sec_finite_comps}
In this subsection, we show that $\mathcal{G}$ contains infinitely many
isolated vertices and finite components with arbitrarily many vertices. To do so,
we prove the following more general result that gives sufficient conditions for a vertex of $\mathcal{G}$ to be contained in a finite component, i.e. a component of $\cG$ that is a finite subgraph. Moreover, the possible gradings of vertices in this component is exactly described. Recall from the introduction that the \textit{height} of a component is the maximal length of a nested sequence of minimal sets in it. 

\begin{prop} \label{thm:isolated_comps}
	Let $B \subset X$ be a box such that: 
	\begin{enumerate}
		\item $\Exc(B) = d \ge 0$
		\item The modulus $\{\alpha, \beta\}$ of $B$ satisfies $\alpha, \beta \ge 2$
		\item Given any standard line $L$, $B \cap L$ is either empty or contains at least $d+2$ vertices. 
		\item $B$ is a dead set
	\end{enumerate}
	Let $\mathcal{C}$ be the component of $\mathcal{G}$ containing the vertex
  representing $B$. 
	Then, for any vertex in $\mathcal{C}$ representing a set $A$, we have that $|B| - d \le |A| \le |B|$ and that $\enc(A)$ is congruent to  $B$. 
	Furthermore, $\mathcal{C}$ contains a  vertex representing a set of size $k$ for every $|B| - d \le k \le |B|$.
	In particular, $\mathcal{C}$ is finite and has height exactly $d+1$.

\end{prop}
\begin{proof}
	We first prove the final claim of the theorem.
	Let $B = B_0 \supset B_1 \supset \dots \supset B_n$ be the sets given by  Lemma~\ref{lemma_exc_bound} where $|\partial B_i| = |\partial B|$ and $n = |L \cap B| - 1 \ge d+1$ (where $L$ is an extremal line of $B$). As $\Exc(B) = d$, we have that $B_i$ is minimal for each $i \le d$. The claim follows.
	
	Now let $\mathcal{C}'$ be the set of all vertices in $\mathcal{C}$ represented by a set $C$ such that there exists a path $C = C_0, \dots, C_n = B$ in $\mathcal{G}$ with $n \le d$ and $|C_{i+1}| = |C_i| + 1$ for all $0 \le i < n$. 
	As $\Exc(B) = d$, for all $C \in \mathcal{C}'$ we must have that $|\partial C| = |\partial B|$ and, consequently by Lemma~\ref{lem:nested_sequence_to_enc_box}, we have that $\enc(C) = B$. Thus, to prove the remaining claims of the theorem, it is enough to show that $\mathcal{C}' = \mathcal{C}$. Additionally, as $\mathcal{C}$ is connected, it is enough to show that given any vertex $v$ in $\mathcal{G}$, represented by a set $A$, that is adjacent to a vertex in $\mathcal{C}'$, represented by a set $C$, then $v \in \mathcal{C}'$. Let $A$ and $C$ be such sets.
	
	Suppose first that $A \subset C$. Then $|A|  = |C| -1 \ge |B| - d - 1$ (by the definition of $\mathcal{C}'$) and by (3) it follows that $A$ contains a vertex in every standard line which has non-empty intersection with $B$. Thus $\enc(A) = B$. 
	However, by Theorem~\ref{thm:minimal set characterization}, we must have that $|A| \ge |B| - d$. Consequently, $v \in \mathcal{C}'$. On the other hand, suppose that $C \subset A$. As $\enc(C) = B$ and as $B$ is dead, we must also have that $v \in \mathcal{C}'$. Thus, $\mathcal{C}' = \mathcal{C}$ as claimed.
\end{proof}



\begin{thm} \label{cor:large_comps}
	The graph $\mathcal{G}$ has finite components of arbitrarily large height and it contains infinitely many isolated vertices.
\end{thm}
\begin{proof}

	Let $\alpha = 2l^3 + l^2 + l$ and $\beta = 2l^3 + l^2 - l$ for some integer $l \ge 3$.
	 Note that $\alpha$ and $\beta$ are always positive.
	Consider the box $B = B(\alpha, \beta)$. Then $B$ is not congruent to the box $B(m-1, m+1)$ for any odd integer $m$, and $B$ is not congruent to a Wang--Wang box. 
	Therefore by Theorem~\ref{thm:dead_char2}, $B$ is a dead set.
	
	By Proposition~\ref{thm:box_excess}, we have that $\Exc(B) = l^3$. 
	Furthermore, given any standard line $L$, $B \cap L$ is either empty or contains at least $\min(\frac{\alpha}{2} + 1, \frac{\beta}{2}+1 ) = l^3 + \frac{l^2}{2}
  - \frac{l}{2} +1$  vertices.  
	In particular, as $l \ge 3$, $B \cap L$ is either empty or contains at least
  $l^3 +2 = \Exc(B) +2$ vertices. 
	Thus by Proposition~\ref{thm:isolated_comps}, the component of $\mathcal{G}$
  containing $B$ is finite and has height 
  at least $l^3$. As this is
  true for any $l \ge 3$, the first claim follows.

  Now let $\alpha = k^2 + k$ and $\beta = k^2 - k$ for some integer $k \ge 4$.
	Then we claim that the box $B = B(\alpha, \beta)$ is an isolated vertex of $\mathcal{G}$. 
	By Theorem~\ref{thm:box_excess}, $\Exc(B) = 0$.
	As above, for any standard line $L$, $B \cap L$ is either empty or contains at least $2$ vertices.
	Also $B$ is a dead set  by Theorem~\ref{thm:dead_char2}.
	Therefore by Proposition~\ref{thm:isolated_comps}, $B$ is an isolated vertex of $\mathcal{G}$. Thus $\mathcal{G}$ contains infinitely many isolated vertices. 
  \end{proof}


\bibliography{refs}

\providecommand{\bysame}{\leavevmode\hbox to3em{\hrulefill}\thinspace}
\providecommand{\MR}{\relax\ifhmode\unskip\space\fi MR }
\providecommand{\MRhref}[2]{%
  \href{http://www.ams.org/mathscinet-getitem?mr=#1}{#2}
}
\providecommand{\href}[2]{#2}
\begin{thebibliography}{HLW06}

\bibitem[BE18]{barbererde}
Ben Barber and Joshua Erde, \emph{Isoperimetry in integer lattices}, Discrete
  Analysis \textbf{7} (2018).

\bibitem[Ber67]{bernstein}
A.~J. Bernstein, \emph{Maximally connected arrays on the {$n$}-cube}, SIAM J.
  Appl. Math. \textbf{15} (1967), 1485--1489.

\bibitem[BL91a]{bollobasleader91}
B\'{e}la Bollob\'{a}s and Imre Leader, \emph{Compressions and isoperimetric
  inequalities}, J. Combin. Theory Ser. A \textbf{56} (1991), no.~1, 47--62.

\bibitem[BL91b]{bollobasleader91edge}
\bysame, \emph{Edge-isoperimetric inequalities in the grid}, Combinatorica
  \textbf{11} (1991), no.~4, 299--314.

\bibitem[Chv75]{chvatalova}
Jarmila Chv\'{a}talov\'{a}, \emph{Optimal labelling of a product of two paths},
  Discrete Math. \textbf{11} (1975), 249--253.

\bibitem[Har64]{harper64}
L.~H. Harper, \emph{Optimal assignments of numbers to vertices}, J. Soc.
  Indust. Appl. Math. \textbf{12} (1964), 131--135.

\bibitem[Har66]{harper66}
\bysame, \emph{Optimal numberings and isoperimetric problems on graphs}, J.
  Combinatorial Theory \textbf{1} (1966), 385--393.

\bibitem[Har76]{hart}
Sergiu Hart, \emph{A note on the edges of the {$n$}-cube}, Discrete Math.
  \textbf{14} (1976), no.~2, 157--163.

\bibitem[Har99]{harper99}
L.~H. Harper, \emph{On an isoperimetric problem for {H}amming graphs},
  Proceedings of the {C}onference on {O}ptimal {D}iscrete {S}tructures and
  {A}lgorithms---{ODSA} '97 ({R}ostock), vol.~95, 1999, pp.~285--309.

\bibitem[Har04]{harper-book}
\bysame, \emph{Global methods for combinatorial isoperimetric problems},
  Cambridge Studies in Advanced Mathematics, vol.~90, Cambridge University
  Press, Cambridge, 2004.

\bibitem[HLW06]{hoory2006expander}
Shlomo Hoory, Nathan Linial, and Avi Wigderson, \emph{Expander graphs and their
  applications}, Bull. Amer. Math. Soc. (N.S.) \textbf{43} (2006), no.~4,
  439--561.

\bibitem[Lin64]{lindsey}
John~H. Lindsey, II, \emph{Assignment of numbers to vertices}, Amer. Math.
  Monthly \textbf{71} (1964), 508--516.

\bibitem[Mog83]{moghadam}
H.S. Moghadam, \emph{Compression operators and a solution to the bandwidth
  problem of the product of $n$ paths}, Ph.D. Thesis, University of California,
  Riverside (1983).

\bibitem[Sie08]{sieben2008polyominoes}
N\'{a}ndor Sieben, \emph{Polyominoes with minimum site-perimeter and full set
  achievement games}, European J. Combin. \textbf{29} (2008), no.~1, 108--117.

\bibitem[VB08]{vainsencher2008isop}
Daniel Vainsencher and Alfred~M. Bruckstein, \emph{On isoperimetrically optimal
  polyforms}, Theoret. Comput. Sci. \textbf{406} (2008), no.~1-2, 146--159.

\bibitem[VR12]{radcliffeveomett12}
Ellen Veomett and A.~J. Radcliffe, \emph{Vertex isoperimetric inequalities for
  a family of graphs on {$\mathbb{Z}^k$}}, Electron. J. Combin. \textbf{19}
  (2012), no.~2, Paper 45, 18.

\bibitem[WW77]{wangwang77}
Da~Lun Wang and Ping Wang, \emph{Discrete isoperimetric problems}, SIAM J.
  Appl. Math. \textbf{32} (1977), no.~4, 860--870.

\end{thebibliography}
\bibliographystyle{amsalpha}
\end{document}